\documentclass[12pt]{amsart}
\usepackage{amsmath}
\usepackage{amsthm}
\usepackage{aliascnt}
\usepackage{amssymb}
\usepackage{graphicx}
\usepackage{crimson}
\usepackage{cochineal}
\usepackage{tikz}
\usepackage{geometry}
\usepackage[dvipsnames]{xcolor}
\usepackage [utf8]{inputenc}

\usetikzlibrary{arrows.meta,calc,decorations.pathreplacing,positioning}

\definecolor{branchblue}{RGB}{35,91,168}
\definecolor{branchorange}{RGB}{201,83,35}
\definecolor{targetgreen}{RGB}{25,125,112}
\definecolor{softgray}{RGB}{105,112,120}

\usepackage{tcolorbox}

\newtcolorbox{commentbox}{
  colback=yellow!10,
  colframe=red!50!black,
  boxrule=0.5pt,
  arc=2pt,
  left=4pt,right=4pt,top=2pt,bottom=2pt
}

\newcommand{\eps}{\varepsilon}
\newcommand{\Id}{\operatorname{Id}}

\newcommand{\R}{\mathbb{R}}
\newcommand{\T}{\mathbb{T}}

\newcommand{\cP}{\mathcal{P}}
\newcommand{\cO}{\mathcal{O}}

\newcommand{\cT}{\mathcal{T}}

\newcommand{\ehc}{\operatorname{Ehc}}
\newcommand{\nuh}{\operatorname{Nuh}}
\newcommand{\inte}{\operatorname{int}}

\newcommand{\diam}{\text{\rm diam}}
\newcommand{\per}{\operatorname{Per}}
\newcommand{\peri}{\operatorname{per}}

\newcommand{\zeroeq}{\stackrel{\rm{o}}{=}}
\newcommand{\transv}{\pitchfork}

\newcommand{\cR}{\mathcal{R}}
\newcommand{\cC}{\mathcal{C}}
\newcommand{\oo}{\infty}
\newcommand{\lam}{\lambda}
\newcommand{\MME}{\mu_{\mathrm{MME}}}
\newcommand{\bp}{{\it Proof.}\quad}
\newcommand{\htop}{\mathrm{h}_{top}}

\newcommand{\epr}{\begin{flushright}\fbox{\hspace{.1em}}\:\end{flushright}}

\newtheorem{theorem}{Theorem}[section]

\newaliascnt{claim}{theorem}
\newtheorem{claim}[claim]{Claim}
\aliascntresetthe{claim}

\newaliascnt{definition}{theorem}
\newtheorem{definition}[definition]{Definition}
\aliascntresetthe{definition}

\newaliascnt{proposition}{theorem}
\newtheorem{proposition}[proposition]{Proposition}
\aliascntresetthe{proposition}

\newaliascnt{lemma}{theorem}
\newtheorem{lemma}[lemma]{Lemma}
\aliascntresetthe{lemma}

\newaliascnt{conjecture}{theorem}

\aliascntresetthe{conjecture}

\newaliascnt{question}{theorem}

\aliascntresetthe{question}

\newtheorem*{question*}{Question}

\newaliascnt{corollary}{theorem}
\newtheorem{corollary}[corollary]{Corollary}
\aliascntresetthe{corollary}

\newaliascnt{remark}{theorem}
\newtheorem{remark}[remark]{Remark}
\aliascntresetthe{remark}

\newtheorem{maintheorem}{Theorem}

\numberwithin{equation}{section}

\usepackage{hyperref}
\usepackage[capitalise,nameinlink]{cleveref}

\crefname{theorem}{theorem}{theorems}
\Crefname{theorem}{Theorem}{Theorems}

\crefname{claim}{claim}{claims}
\Crefname{claim}{Claim}{Claims}

\crefname{definition}{definition}{definitions}
\Crefname{definition}{Definition}{Definitions}

\crefname{proposition}{proposition}{propositions}
\Crefname{proposition}{Proposition}{Propositions}

\crefname{lemma}{lemma}{lemmas}
\Crefname{lemma}{Lemma}{Lemmas}

\crefname{conjecture}{conjecture}{conjectures}
\Crefname{conjecture}{Conjecture}{Conjectures}

\crefname{question}{question}{questions}
\Crefname{question}{Question}{Questions}

\crefname{corollary}{corollary}{corollaries}
\Crefname{corollary}{Corollary}{Corollaries}

\crefname{remark}{remark}{remarks}
\Crefname{remark}{Remark}{Remarks}

\crefname{maintheorem}{theorem}{theorems}
\Crefname{maintheorem}{Theorem}{Theorems}

\definecolor{cobalt}{rgb}{0.0, 0.28, 0.67}
\hypersetup{
    colorlinks=true,
    linkcolor=cobalt,
    citecolor=cobalt
    }

\title{Entropy and semiconjugacy on surfaces}

\author{Pablo D. Carrasco}
\address{ICEx-UFMG, Avda. Presidente Antonio Carlos 6627, Belo Horizonte-MG, BR31270-90}
\email{pdcarrasco@ufmg.br}
\thanks{P.D. Carrasco was partially supported by CNPq (Produtividade em Pesquisa) 305197/2025-8, Projeto Universal CNPq 401737/2025-0, FAPEMIG APQ-02338-24, and CAPES through the MATH-AMSUD program.}

\author{Federico Rodriguez-Hertz}
\address{Penn State, 227 McAllister Building, University Park, State College, PA16802}
\email{hertz@math.psu.edu}
 \thanks{F. Rodriguez-Hertz was partially supported by NSF grant DMS-2453688.}

\author{Jana Rodriguez-Hertz}
\author{Raúl Ures}
\address[J. Rodriguez-Hertz and R. Ures]
{ Department of Mathematics and Shenzhen International Center for Mathematics - Southern University of Science and Technology,
1088 Xueyuan Avenue, Shenzhen 518055, P.R. China}
\email[J. Rodriguez-Hertz]{rhertz@sustech.edu.cn}
\email[R. Ures]{ures@sustech.edu.cn}

 \date{\today}



\begin{document}

\begin{abstract}
Let $g$ be a $C^\infty$ diffeomorphism in the isotopy class of a
pseudo-Anosov homeomorphism $f$ such that $g$ and $f$ have the same
topological entropy. In 1988, Handel proved that this implies the
existence of a semiconjugacy $\pi$ from $g$ to $f$. He
stated that, in general, there is at least one point $x$ such that
$\pi^{-1}(x)$ is disconnected. We show that this is not the case: for every $x$, the set
$\pi^{-1}(x)$ is the intersection of a nested sequence of closed topological disks, and hence is connected.\par
We also prove that there is a unique $g$-invariant probability measure projecting to the measure of maximal entropy of $f$. This measure is entropy-maximizing, hyperbolic, and Bernoulli, and the semiconjugacy induces a metric isomorphism between the corresponding measure-preserving systems.
\end{abstract}

 \maketitle

 \section{Introduction}
\label{sec:introduction}
In this paper we study properties of smooth surface diffeomorphisms that are isotopic to a pseudo-Anosov map. Within the Nielsen–Thurston classification of surface homeomorphisms (periodic, reducible, and pseudo-Anosov), it is the pseudo-Anosov type that exhibits the richest dynamical and geometric structure. See \cite{Thurston_1988}, \cite{Fathi_2021}.

The best understood case occurs when the surface is the torus $\T^2$. Every pseudo-Anosov isotopy class on $\T^2$ is represented by an Anosov diffeomorphism. By a classical result of Franks \cite{franksthe}, if $g$ is in the isotopy class of an Anosov diffeomorphism $f$, then $g$ is semiconjugate to $f$; that is, there exists a continuous surjection $\pi:M\to M$ such that $\pi g=f\pi$;  we also say that $f$ is a factor of $g$.
 Thus, the linear Anosov map realizes the minimum of the topological entropy in its isotopy class.

When the genus of the surface is greater than one, it is not necessarily true that a map $g$ isotopic to a pseudo-Anosov homeomorphism $f$ is semiconjugate to $f$; new ``large dynamics'' may appear, preventing the existence of the factor map. By Fathi and Shub \cite{FS79}, the topological entropy of $g$ is at
least as large as that of $f$; thus, equality is the extremal case in
this isotopy class. This phenomenon was elucidated in a seminal paper by Handel \cite{handel1988}, who proved that if the topological entropies of $g$ and $f$ coincide, 
then there exists a semiconjugacy $\pi$ from $g$ to $f$.
The entropy condition above prevents the appearance of large dynamics, but of course small invariant sets carrying non-trivial dynamics can appear. At first sight, one might think that the fibers of the semiconjugacy could be quite wild. In his pioneering work, Handel reinforces this expectation by stating:
\begin{center}
\begin{minipage}[t]{.8\textwidth}
    ``...if $f: M\to M$ is a pseudo-Anosov diffeomorphism and if
$g: M \to M$ is homotopic to $f$, then there is a closed subset $Y\subset M$ such that $g|_Y$ is
semiconjugate to $f$ ; i.e. there is a map $\pi:Y\to M$ such that $\pi\circ g|_Y =f\circ \pi$. In this
paper we give a sufficient condition for $Y$ to be all of $M$ and hence for $g$ to be
semiconjugate to $f$. In general the pre-images of points under the semi-conjugacy
{\bf will be disconnected}.''
\end{minipage}
\end{center}
\medskip
In this paper we show that the phenomenon anticipated by Handel does not
occur. In fact, every fiber of
$\pi$ is {\em cellular}, that is, it is the intersection of a nested sequence of closed
topological disks.\par
\smallskip

\begin{maintheorem}\label{mainthm:fiber}
Let $g$ be a $C^{\infty}$ diffeomorphism of a closed surface $M$. Suppose that $g$ is isotopic to a pseudo-Anosov map $f$, with $\htop(f)=\htop(g)$, and let $\pi:M\to M$ be the semiconjugacy from $g$ to $f$ given by Handel's theorem. 
Then $\pi^{-1}(x)$ is cellular for every $x\in M$.
\end{maintheorem}
\smallskip


\subsection{The relevant measure of maximal entropy}


For a probability measure $\mu$ on $M$, let $\pi_*\mu$ denote its
pushforward by $\pi$.\par
Let $\mu_{\mathrm{MME}}$ denote the unique measure of maximal entropy for the pseudo-Anosov map $f$.  
\begin{maintheorem}\label{mainthm:isomorphism.mu} 
Let $g$ be a $C^{\infty}$ diffeomorphism of a closed surface $M$, and suppose that $g$ is isotopic to a pseudo-Anosov map $f$.
Suppose that
\[
\htop(f)=\htop(g),
\]
and let $\pi:M\to M$ be the semiconjugacy from $g$ to $f$ given by Handel's theorem.

Then there exists a unique \(g\)-invariant probability measure $\mu$
such that
\[
\pi_*\mu=\MME.
\]
This measure is an entropy-maximizing measure for \(g\). Moreover, it is hyperbolic and Bernoulli, and
\[
\pi:(M,\mu,g)\longrightarrow(M,\MME,f)
\]
is a metric isomorphism.
\end{maintheorem}

The existence of a {\em measure-theoretic isomorphism} between two Ber\-nou\-lli shifts of equal entropy has been known since Ornstein's theorem \cite{ornstein1970bernoulli}.

 Theorem \ref{mainthm:isomorphism.mu} shows that the topological semiconjugacy $\pi$ {\em is itself an isomorphism} between the corresponding Bernoulli systems.  Thus, $\pi$ realizes a concrete, dynamically defined isomorphism whose abstract existence follows from Ornstein’s theory.

\subsection{Outline of the proof}
\label{sec:outline}
\begin{enumerate}
\item The proof of \Cref{mainthm:fiber} relies on \Cref{mainthm:isomorphism.mu}: we use that $\pi$ is monotone on almost all stable and unstable manifolds. 
For any given point $x$, we construct a neighborhood basis of $\pi(x)$
consisting of polygons whose boundaries are formed by stable and unstable
manifolds of $f$. Their sides are the $\pi$-images of segments contained
in stable and unstable manifolds of $g$ on which $\pi$ is monotone. The
corresponding lifted polygons form a neighborhood basis of
$\pi^{-1}(\pi(x))$.\par
That is, we construct adapted polygons with stable and unstable sides around an arbitrary point $\pi(x)$, and show that their
lifts trap the entire corresponding fiber $\pi^{-1}(\pi(x))$. By choosing a nested sequence of such polygons, we conclude that every fiber of $\pi$ is
a nested intersection of topological disks. 
 The proof of this involves the topology of the plane and an order argument. 
See \Cref{ssub:adaptedpolygons}.
\item  To show that $\pi$ is monotone on almost every invariant manifold, we see that the equality of entropy implies that the extension defined by $\pi$ contains no additional entropy inside the fibers. When this is expressed through partitions subordinate to unstable manifolds, the common refinement of the subordinate partition and the fiber partition is essentially atomic. Applying the same argument to $g^{-1}$ gives the corresponding statement along stable manifolds.  This implies the existence of a full-measure set $X$ such that no two distinct points of $X$ lying on the same stable or unstable manifold have the same image under $\pi$. Positive entropy implies that $\pi$ is onto the invariant manifolds of points of $X$. \par
To obtain the monotonicity of $\pi$ along invariant manifolds, one has to
rule out \emph{folds}, that is, pairs of points $x,y$ on the same
invariant manifold satisfying $\pi(x)=\pi(y)$ and having an intermediate
point $x<z<y$ such that $\pi(z)\neq\pi(x)=\pi(y)$. 
A fold produces a Jordan domain in the universal cover, partially bounded by an unstable segment of a typical point $x$. Let $\mu$ be the measure of maximal entropy obtained in Theorem \ref{mainthm:isomorphism.mu}. By Ben Ovadia's result \cite{Ovadia2023}, $\mu$ has product structure. The structure of lifted invariant manifolds implies that there is a positive measure set of points whose stable manifolds topologically cross the folded unstable manifold twice. 

We use a version of Sard's lemma developed in \cite{HHTUSRB}, later adapted to measures of maximal entropy in \cite{BCS2022}, that implies that the stable manifolds transversely intersect the unstable manifold of $x$ twice. This produces a stable holonomy between two parts of the folded unstable segment. The two holonomy-related points have the same $\pi$-image because their images lie simultaneously on the same stable and unstable leaves of $f$, whose lifts intersect uniquely. Ben Ovadia's product structure implies that stable holonomy sends sets of positive conditional measure in one unstable segment into sets of positive conditional measure in the other unstable segment, thus producing a set of positive conditional measure on which $\pi$ is not injective, a contradiction. This provides the required monotonicity of $\pi$ when restricted to invariant manifolds, which is needed in Step (1). See Section \ref{sec:isomorphism}. 

\item The fundamental part of the proof of \Cref{mainthm:isomorphism.mu} is to show that there exists essentially one ergodic homoclinic class. {\em Ergodic homoclinic classes} were developed by F.~Rodriguez Hertz, J.~Rodriguez Hertz, A.~Tahzibi and R.~Ures in \cite{HHTU11} (see also \cite{HHTUSRB}). They are the measurable analogs of the homoclinic classes (basic sets) appearing in the Smale spectral decomposition theorem of Axiom A diffeomorphisms \cite{Smale_1967}. Namely, given a periodic point $p$, the {\em ergodic homoclinic class} of $p$ is the set of points $x$ such that
$$W^+(x)\transv W^-(\cO(p))\ne\varnothing\qquad \text{and}\qquad W^-(x)\transv W^+(\cO(p))\ne\varnothing,$$
where $W^-(x)$ and $W^+(x)$ denote the Pesin stable and unstable manifolds defined in Section \ref{sec:preliminaries}. An important part of showing that there is essentially one ergodic homoclinic class is proving that the $\pi$-images of stable and unstable manifolds are unbounded in the universal cover for all points in a full-measure set. See Section \ref{sec:surjective.invariant}. From this it follows that for almost every $x$ and almost every $y$, the stable manifold of $x$ and the unstable manifold of $y$ intersect in a topologically transverse way. See Section \ref{section.topological.intersection}. 
The Sard argument mentioned above allows one to show that there is a full measure set where these intersections are transverse; hence almost all points belong to the same ergodic homoclinic class. Once there is a unique large ergodic homoclinic class, it follows that the measure $\mu$ is Bernoulli by the corresponding spectral-decomposition theorem. See Section \ref{sec:Sard}. 
\end{enumerate}
\begin{remark}[The conservative case]\label{rem:conservative}
If $g$ preserves the area measure $m$ and $\pi_*m=\mu_{\mathrm{MME}}$, then the conclusion of \Cref{mainthm:fiber} remains valid under the weaker assumption that $g$ is $C^r$, $r>1$.\par
More generally, if $\pi_*m$ is ergodic and has positive entropy, then $m$ is hyperbolic and Bernoulli. If, in addition,
\[
h_m(g)=h_{\pi_*m}(f),
\]
then
\[
\pi\colon (M,m,g)\longrightarrow (M,\pi_*m,f)
\]
is a metric isomorphism.
\end{remark}
\subsection{Relation with entropy-preserving models.}
Theorems \ref{mainthm:fiber} and \ref{mainthm:isomorphism.mu} also fit into the broader search for simple models that retain the essential dynamics of a system. 
An early formulation of the relation between entropy and stability appears in
Shub's note \cite{Shu75}, which places
homotopy-invariant lower bounds for topological entropy in a stability
framework. For pseudo-Anosov maps, Fathi and Shub obtained the sharp form
relevant here: a pseudo-Anosov homeomorphism realizes the minimum of
topological entropy in its isotopy class \cite{FS79}. Thus, the entropy equality
considered in this paper is the extremal case of a topological lower bound that
is fixed throughout the isotopy class.
In one-dimensional dynamics, Milnor--Thurston kneading theory shows that a unimodal interval map of positive entropy admits a monotone semiconjugacy to a tent map with the same entropy \cite{MT88}. A two-dimensional counterpart was obtained by Crovisier and Pujals \cite{CP18}: every strongly dissipative diffeomorphism of the disk admits a semiconjugacy to a continuous map of a compact real tree; this semiconjugacy preserves the entropy of every non-atomic ergodic measure and distinguishes such measures. In the zero-entropy regime, Crovisier, Pujals and Tresser \cite{CPT24} prove that mildly dissipative disk diffeomorphisms are either generalized Morse--Smale or infinitely renormalizable. Their aperiodic chain-recurrence classes factor onto odometers through semiconjugacies whose fibers are the connected components of the class and are singletons almost everywhere. The one-dimensional character of this picture is further developed by Crovisier, Lyubich, Pujals and Yang \cite{CLPY24}, who show, for a class of infinitely renormalizable unicritical disk diffeomorphisms, that the renormalizations converge super-exponentially to the space of one-dimensional unimodal maps.\par

The real-tree model of Crovisier and Pujals is particularly close in spirit to
an earlier construction of Fathi for pseudo-Anosov dynamics \cite{Fat90}. The
spaces of leaves of the lifted stable and unstable foliations, endowed with the
metrics induced by their transverse measures and then completed, are real
trees; their product carries a hyperbolic extension of the pseudo-Anosov map.
In unpublished work of three of the present authors \cite{HHU04}, this
construction was used to compare all the dynamics in the isotopy class of a
fixed pseudo-Anosov map. More precisely, for each \(g\) in this isotopy class,
one identifies points whose lifted \(g\)-orbits remain at bounded distance for
all \(n\in\mathbb Z\). The resulting quotient dynamics is conjugate to the
restriction of the same hyperbolic homeomorphism to an invariant subset of the
product of the two completed trees.

This comparison also highlights the particular strength of the construction
of Crovisier and Pujals \cite{CP18}: strong dissipation produces a single
compact real tree as a genuine one-dimensional factor of the original
dynamics, and the resulting reduction loses neither non-atomic ergodic
measures nor their entropy. The Fathi--HHU construction  provides a
two-sided hyperbolic ambient model, built from the product of the stable and
unstable trees, for the entire pseudo-Anosov isotopy class. \par

Our setting is not the same:
we do not assume a dissipation hypothesis, but the quotient dynamics is prescribed a priori by the pseudo-Anosov representative of the isotopy class, which can be further quotiented to one-dimensional dynamics.  Under the equality of topological entropies, Handel's theorem provides the factor map. 
Theorems \ref{mainthm:fiber} and \ref{mainthm:isomorphism.mu} show that this factor is faithful in two complementary senses: all its fibers are cellular, while for the distinguished measure of maximal entropy, the semiconjugacy is a metric isomorphism. More generally, every invariant measure possesses an upstairs counterpart where the factor map collapses only structures of strictly lower entropy. Conjecturally, any collapsed positive entropy is localized entirely on periodic orbits.

\subsection{Relation with higher-rank rigidity.}
The strategy of the proof has a higher-rank antecedent in a series of papers by A. Katok and F. Rodriguez Hertz. For actions with Cartan homotopy data, \cite{KRH2007} established uniqueness of the invariant measure projecting to Haar measure, measurable isomorphism under the semiconjugacy, monotonicity properties, and cellularity of its fibers. The measure-theoretic and holonomy mechanisms were further developed in \cite{KRH2010}, while the passage from measurable rigidity to topology through boxes bounded by codimension-one invariant manifolds was developed in \cite{KRH2016}. Our argument has the same broad structure: measurable rigidity first yields leafwise monotonicity, which is then used to trap fibers inside adapted polygons. The essential difference is that the affine structures and commuting elements supplied by higher rank are unavailable for a single surface diffeomorphism; their role is played here by entropy equality, homoclinic-class theory, product structure, and the planar geometry of the stable and unstable leaves of the pseudo-Anosov factor.

\section{Preliminaries}
\label{sec:preliminaries}

Given a measure $\mu$ and a measurable map $T:X\to Y$, $T_*\mu$ denotes the pushforward of $\mu$ under $T$; that is, the measure that satisfies for all measurable sets $B$:
\[
T_*\mu(B)=\mu(T^{-1}(B)).
\]
If $A$ and $B$ are measurable sets, $A\zeroeq B$ means that $\mu(A\triangle B)=0$. If $T:X\to X$, 
we say that $\mu$ is $T$-invariant if $T_*\mu=\mu$. We say that a $T$-invariant measure is {\em ergodic} if $T^{-1}(A)\zeroeq A\Rightarrow \mu(A)\mu(A^c)=0$. 

If $I$ is a measurable set satisfying $\mu(I)>0$, we denote by $\mu_I$ the conditional measure on $I$, that is 
\[
\mu_I(A)=\frac{\mu(A\cap I)}{\mu(I)}\qquad \forall A\text{ measurable}.    
\]
Note that if $\mu$ is $T$-invariant and $I$ is a $T$-invariant set (meaning, $I\zeroeq T^{-1}I$), then $\mu_I$ is $T$-invariant.

Assume $T:X\to X$ is a continuous map of a compact metric space. By the variational\- principle (see for example \cite{Mane1987}), its topological entropy satisfies
\[
\htop(T)=\sup\bigl\{h_\mu(T):\mu\text{ is a $T$-invariant probability measure}\bigr\},
\]
where $h_\mu(T)$ denotes the metric entropy of $T$ with respect to $\mu$. A $T$-invariant probability measure $\mu$ is called a {\em measure of maximal entropy} if
\[
h_\mu(T)=\htop(T).
\]
The supremum in the variational principle is not always attained.

Let $g: M\to M$ be a $\cC^{r}$ diffeomorphism of a compact Riemannian surface, $r\geq 1$, and let $\mu$ be a $g$-invariant probability measure. For $\mu$-almost every $x$, the {\em Lyapunov exponents} of $g$ at $x$ are given by
\begin{align*}
&\lambda^{+}(x;g)=\lim_{n\to+\infty}\frac{1}{n}\log \lVert Dg_x^n\rVert\\
&\lambda^{-}(x;g)=-\lim_{n\to+\infty}\frac{1}{n}\log \bigl\lVert (Dg_x^n)^{-1}\bigr\rVert.
\end{align*}
It follows that $\lambda^{-}(x)\leq \lambda^{+}(x)$. If $\mu$ is ergodic, the Lyapunov exponents are constant $\mu$-almost everywhere.

Whenever these exponents are distinct, Oseledet's theorem implies the existence of a $Dg$-invariant splitting
\[
T_xM=E^{-}(x)\oplus E^{+}(x)
\]
such that, for every nonzero $v\in E^{\pm}(x)$,
\[
\lim_{n\to\pm\infty}\frac{1}{n}\log \lVert Dg_x^n v\rVert=\lambda^{\pm}(x).
\]
The measure $\mu$ is called {\em hyperbolic} if 
\[
\lambda^-(x)<0<\lambda^+(x)
\]
for $\mu$-almost every $x$.

For every $x\in M$, define the {\em Pesin stable manifold} of $x$ by
\[
W^-(x;g)=\left\{y:\limsup_{n\to+\infty}\frac1n\log d(g^n(x),g^n(y))< 0\right\}.
\]
The {\em Pesin unstable manifold} $W^+(x;g)$ of $x$ is $W^-(x;g^{-1})$. If $r>1$ there is a Borel set $\mathcal R$ contained in the non-wandering set $NW(g)$ of $g$, such that for $x\in\mathcal R$, $W^\pm(x;g)$ is an immersed manifold \cite{pesin76}, and $\mu(\mathcal R)=1$ for every $g$-invariant hyperbolic measure $\mu$.

For $x\in\mathcal R$ and $\eps>0$ we denote by
\[
    W^\pm_{\eps}(x;g)
\]
the interval inside $W^\pm(x;g)$, centered at $x$ and of radius $\eps$. We fix some $\rho>0$ small (say, less than $10^{-2}$ times the injectivity radius of $\exp$), and denote
\[
    W^\pm_{\mathrm{loc}}(x;g)
\]
the connected component of $W^\pm(x;g)\cap B_\rho(x)$, where $B_\rho(x)$ is the Riemannian ball of center $x$ and radius $\rho$.

For a pseudo-Anosov homeomorphism $f$, we define
\[
W^s(x;f)=\left\{y\in M:d(f^n(x),f^n(y))\longrightarrow 0 \text{ as }n\longrightarrow+\infty\right\},
\]
and define $W^u(x;f)=W^s(x;f^{-1})$.\par
If $\Lambda$ is a hyperbolic set contained in the non-wandering set $NW(f)$ of a diffeomorphism $f$, then $\Lambda\subset\cR$ and for every $x\in\Lambda$
$$W^s(x;f)=W^-(x;f)\qquad\text{and}\qquad W^u(x;f)=W^+(x;f).$$

\subsection{Ergodic homoclinic classes and spectral decompositions.}
\label{subsection.pesin.spectral}

In the sixties, Smale discovered a decomposition of hyperbolic dynamics into basic pieces with strong recurrent properties. See Theorem \ref{teo:smale}.
Each of these basic sets consists of a {\em homoclinic class}, which we define below. \par
Given a diffeomorphism $f:M\to M$, for any two points $x,y\in \cR$ we define the relation $x\approx y$ if 
$$W^-(x;f)\transv W^+(y;f)\ne\varnothing\qquad\text{and}\qquad W^+(x;f)\transv W^-(y;f)\ne\varnothing.$$
We also define the relation $x\sim y$ if there is $z\in\cO(x)$ such that $z\approx y$.
When restricted to the set of hyperbolic periodic points, both $\approx$ and $\sim$ are equivalence relations. Let $\per_h(f)$ denote the set of hyperbolic periodic points of $f$. We define the {\em homoclinic class} of $p$ by
$$HC(p)=\overline{\{p'\in \per_h(f):p'\sim p\}}, $$
and the {\em small homoclinic class} of $p$ by
$$HC^*(p)=\overline{\{p'\in \per_h(f):p'\approx p\}}.$$
\par
In 2011, F. Rodriguez Hertz, J. Rodriguez Hertz, A. Tahzibi and R. Ures extended this notion to a more general one, depending only on Pesin stable and unstable manifolds. \par
Given a hyperbolic periodic point $p$, its {\em ergodic homoclinic class} is defined by
$$\ehc(p)=\{x\in\cR: x\sim p\}.$$
The {\em small ergodic homoclinic class} of $p$ is defined by
$$\ehc^*(p)=\{x\in\cR:x\approx p\}.$$

Recall that a diffeomorphism $f$ is {\em Axiom A} if its non-wandering set $NW(f)$ is hyperbolic and periodic points are dense in $NW(f)$.
The following proposition follows from the Inclination Lemma and is left to the reader.
\begin{proposition} If $f$ is Axiom A, then for every periodic point $p\in NW(f)$,
$$\ehc(p)=HC(p)\qquad\text{and}\qquad \ehc^*(p)=HC^*(p)$$
\end{proposition}

Smale's spectral decomposition theorem \cite{Smale_1967} states that the
non-wandering set of an Axiom~A diffeomorphism decomposes into finitely
many disjoint homoclinic classes.

\begin{theorem}[Smale Spectral Decomposition theorem \cite{Smale_1967}]\label{teo:smale} If $f$ is an Axiom A diffeomorphism, then there exist hyperbolic periodic points $p_1,\dots,p_N$ such that
     $$NW(f)=\bigcup_{n=1}^N \ehc(p_n)\quad\text{is a disjoint union}.$$
Moreover:
\begin{enumerate}
    \item $f$ is transitive on each $\ehc(p_n)$.
    \item $f^{\peri(p_n)}$ is topologically mixing on each $\ehc^*(p_n)$.
\end{enumerate}
\end{theorem}
\begin{remark}
    For a hyperbolic periodic point $p$ in $NW(f)$ with $f$ Axiom A, $\cO(p)\subset \ehc^*(p)$ if and only if $\ehc^*(p)=\ehc(p)$, and in that case $f$ is topologically mixing on $\ehc^*(p)$.
\end{remark}

The ergodic homoclinic classes introduced in \cite{HHTU11} made it possible to establish an analogue of Smale's classical theorem for volume-preserving smooth diffeomorphisms $f$ on the Pesin region $\nuh(f)$. The Pesin region of $f$ is the set of points $x\in M$ such that 
$$\limsup_{n\to\infty}\frac{1}{n}\log\|Df^n(x)v\|\ne0\qquad\forall v\in T_xM\setminus\{0\}.$$
For a reminder of the Bernoulli property see \Cref{sec:Bernoulli}. 

\begin{theorem}[Ergodic Spectral Decomposition Theorem for the volume measure \cite{HHTU11}]\label{teo.ergodic.spectral}
Let $f$ be a $C^r$ diffeomorphism preserving a smooth volume $m$, $r>1$. Then there exist hyperbolic periodic points $p_1,\dots, p_N$ ($N\leq \infty$) such that  
$$\nuh(f)\zeroeq \bigcup_{n=1}^N\ehc(p_n), \qquad m(\ehc(p_i)\cap\ehc(p_j))=0\quad\forall i\ne j$$
Moreover,
\begin{enumerate}
    \item $f$ is ergodic on each $\ehc(p_n)$.
    \item $f^{\peri(p_n)}$ is Bernoulli on each $\ehc^*(p_n).$
\end{enumerate}
\end{theorem}

\begin{remark}\label{rmk:bernoulli}
For a hyperbolic periodic point $p$ such that $m(\ehc(p))>0$, we have $O(p)\subset\ehc^*(p)$ if and only if $\ehc^*(p)\zeroeq\ehc(p)$, and in that case,  $f|_{\ehc(p)}:\left(\ehc(p),m_{\ehc(p)}\right)\to \left(\ehc(p),m_{\ehc(p)}\right)$ is Bernoulli. 
\end{remark}

Theorem \ref{teo.ergodic.spectral} is not necessarily true for general hyperbolic measures, even on surfaces. For a general invariant measure, an ergodic homoclinic class does not determine a unique ergodic measure: 
a single ergodic homoclinic class may be of full measure with respect to two distinct ergodic hyperbolic measures. Such
measures are homoclinically related (see below), but need not coincide. This is the obstruction to obtaining, for arbitrary hyperbolic measures, a measure-independent spectral decomposition analogous to the volume-preserving one.

For surface diffeomorphisms, J.~Buzzi, S.~Crovisier, and
O.~Sarig \cite{BCS2022} define an equivalence relation on ergodic
hyperbolic measures. In terms of the ergodic homoclinic classes
introduced in \cite{HHTU11}, it can be formulated as follows.

Two $g$-ergodic hyperbolic measures $\mu_1$ and $\mu_2$ are {\em homoclinically related} (denoted by $\mu_1\sim \mu_2$) if there exists a hyperbolic periodic point $p$ such that
\begin{equation}\label{def:homoc.related.measures}
    \mu_1(\ehc(p))=\mu_2(\ehc(p))=1.
\end{equation}

This is equivalent to the definition in \cite[Definitions~2.9 and~2.10]{BCS2022}, by \cite[Proposition 2.15]{BCS2022}.\par

We will use the following.

\begin{theorem}[\cite{BCS2022}]\label{thm:BCS}
    Let $g$ be a $C^\infty$ surface diffeomorphism and let $\mu$ be an ergodic hyperbolic measure of maximal entropy for $g$. Then there exists a hyperbolic periodic point $p$ such that $\mu(\ehc(p))=1$, and moreover
    \begin{enumerate}
        \item if $\nu$ is an ergodic entropy-maximizing hyperbolic measure such that $\nu\sim\mu$, then $\nu=\mu$.
        \item $\mu$ is supported in $\overline{\ehc(p)}$. 
        \item 
        $(g,\mu)$ is isomorphic to the product of a Bernoulli scheme and a finite cyclic permutation.

    \end{enumerate}
\end{theorem}
\begin{remark}\label{rmk:bernoulli.mme}
 We have that $\ehc^*(p)\zeroeq\ehc(p)$ if and only if $\cO(p)\subset \ehc^*(p)$, and in that case $(g,\mu)$ is Bernoulli. 
\end{remark}
See \cite[Section 3.5]{BCS2022}.

\subsection{The Sard argument of HHTU}\label{sec:sard}
The following Sard argument was developed in \cite{HHTUSRB} for SRB measures, and is crucial in this paper. We state it here for area, to avoid excess of definitions. 

\begin{theorem}[RHRHTU-Sard's lemma \cite{HHTUSRB}]\label{thm:sard.HHTU} Let $g$ be a $C^{1+\alpha}$-diffeomorphism preserving an area probability measure $m$. Assume that $m(\nuh(g))>0$. Given $x\in\cR$, define:
$$\cC(x)=\{y\in \cR:W^-(y)\transv_{top}W^+(x)\ne\varnothing\},$$
the set of points whose stable manifolds have a topological crossing with the unstable manifold of $x$. See precise definition of $\transv_{top}$ in \Cref{def:topological.crossing}. \par
Let
$$\cT(x)=\{y\in \cC(x): W^-(y)\text{ has a non-transverse intersection with }W^+(x)\}.$$
    Then $$m(\cT(x))=0.$$
\end{theorem}
This theorem will be used to prove the conservative version of \Cref{mainthm:isomorphism.mu}. See \Cref{rem:conservative}.\par
For the proof of \Cref{mainthm:isomorphism.mu}, we will use the adapted version of the theorem above to measures of maximal entropy. 
\begin{theorem}[Adapted Sard's lemma \cite{BCS2022}]\label{thm:sard.BCS}
 Let $g$ be a $C^{\infty}$-diffeomorphism, and let $\mu$ be a measure of maximal entropy of $g$. Assume that $\mu(\nuh(g))>0$. Given $x\in\cR$, define:
$$\cC(x)=\{y\in \cR:W^-(y)\transv_{top}W^+(x)\ne\varnothing\},$$
the set of points whose stable manifolds have a topological crossing with the unstable manifold of $x$. See precise definition of $\transv_{top}$ in \Cref{def:topological.crossing}. \par
Let
$$\cT(x)=\{y\in \cC(x): W^-(y)\text{ has a non-transverse intersection with }W^+(x)\}.$$
Let $d_H$ denote the transverse Hausdorff dimension. Then $$d_H(\cT(x))=0.$$    
\end{theorem}
This follows from \cite[Theorem 4.2]{BCS2022}, by taking increasingly large Pesin blocks in $\cC(x)$. On each Pesin block, the stable leaves $W^-(y)$ form a codimension-one lamination with \(C^\infty\) leaves that is \(C^1\) transversely \cite[pp.~47--48]{HHTUSRB}; consequently, its holonomy maps are \(C^1\), and hence Lipschitz.
\subsection{Pseudo-Anosov factors}
\label{ssec:pseusoanosovfactors}

Let $g$ be a $C^r$ diffeomorphism, $r>1$, which is semiconjugate to a pseudo-Anosov map $f$ by a semiconjugacy $\pi$. Then for every $x$,
\[
\pi(W^+(x;g))\subset W^u(\pi(x);f),\qquad  \pi(W^-(x;g))\subset W^s(\pi(x);f). 
\]
No regularity of $\pi$ beyond continuity is required here.

\smallskip 
   
\noindent\textbf{Notation:} for a surface map $f:M\to M$ we denote by $\widetilde f$ a lift of $f$ to the universal cover $\widetilde M$ of $M$ (which we identify with $\R^2$).

\smallskip 
 
We also include the following proposition for later use. Suppose that $\pi$ is homotopic to the identity. Let
\[
\widetilde\pi:\widetilde M\to\widetilde M
\]
be the lift of $\pi$ obtained by lifting a homotopy from $\Id_M$ to $\pi$, starting with $\Id_{\widetilde M}$. Choose compatible lifts $\widetilde f,\widetilde g$ of $f,g$ such that
\[
\widetilde\pi\circ\widetilde g =\widetilde f\circ\widetilde\pi.
\]

\begin{proposition}\label{prop.semiconj}
For any $x,y\in M$, we have $\pi(x)=\pi(y)$ if and only if there exist lifts $\widetilde x,\widetilde y$ of $x,y$ such that
\[
\sup_{n\in\mathbb Z} d\bigl(\widetilde g^n(\widetilde x),\widetilde g^n(\widetilde y)\bigr)<\infty.
\]
\end{proposition}

\begin{proof}
By uniqueness of the lift, $\widetilde\pi$ is equivariant under deck transformations. Since deck transformations are isometries, the displacement function
\[
\widetilde x\in\widetilde M \longmapsto d\bigl(\widetilde\pi(\widetilde x),\widetilde x\bigr)
\]
is invariant under deck transformations. It therefore descends to $M$ and, by compactness, is uniformly bounded. Thus there exists
$K>0$ such that
\[
d\bigl(\widetilde\pi(\widetilde x),\widetilde x\bigr)\leq K
\]
for every $\widetilde x\in\widetilde M$.

Suppose first that $\pi(x)=\pi(y)$, and choose lifts $\widetilde x,\widetilde y$ such that
\[
\widetilde\pi(\widetilde x)=\widetilde\pi(\widetilde y).
\]
For every $n\in\mathbb Z$, the compatibility of the lifts implies
\[
\widetilde\pi\circ\widetilde g^n=\widetilde f^n\circ\widetilde\pi.
\]
Consequently,
\begin{align*}
d\bigl(\widetilde g^n(\widetilde x),\widetilde g^n(\widetilde y)\bigr)\leq
d\bigl(\widetilde g^n(\widetilde x),\widetilde\pi(\widetilde g^n(\widetilde x))\bigr)+
d\bigl(\widetilde\pi(\widetilde g^n (\widetilde y)),
       \widetilde g^n(\widetilde y)\bigr)\leq 2K.
\end{align*}
Hence, the two lifted orbits remain at uniformly bounded distance.

Conversely, suppose that
\[
C:=\sup_{n\in\mathbb Z}d\bigl(\widetilde g^n(\widetilde x),\widetilde g^n(\widetilde y)\bigr)<\infty.
\]
Then, for every $n\in\mathbb Z$,
\begin{align*}
d\bigl(\widetilde f^n(\widetilde\pi(\widetilde x)),
       \widetilde f^n(\widetilde\pi(\widetilde y))\bigr)
=d\bigl(\widetilde\pi(\widetilde g^n(\widetilde x)),
       \widetilde\pi(\widetilde g^n(\widetilde y))\bigr)\leq C+2K.
\end{align*}
Distinct orbits of a lift of a pseudo-Anosov homeomorphism cannot remain at bounded distance for all $n\in\mathbb Z$
\cite{Handel_1985}. Therefore,
\[
\widetilde\pi(\widetilde x)=\widetilde\pi(\widetilde y),
\]
and hence $\pi(x)=\pi(y)$.
\end{proof}

\begin{remark}
If $g$ is homotopic to the pseudo-Anosov map $f$ and
$h_{\mathrm{top}}(g)=h_{\mathrm{top}}(f)$, then Handel's theorem
\cite{handel1988} gives a semiconjugacy homotopic to the identity.
\end{remark}

\subsection{Large measures}
\label{ssec:largemeasures}
Let $\pi$ be the semiconjugacy from $g$ to $f$ given by Handel \cite{handel1988}. 
\begin{definition}\label{def:largemeasures}
Let $g$ be a map semiconjugate to a pseudo-Anosov map $f$ by a semiconjugacy $\pi$. An invariant measure $\mu$ for $g$ is
\begin{enumerate}
    \item {\em large} if $\pi_*\mu$ is ergodic for $f$ and $h_{\pi_*\mu}(f)>0$,
    \item {\em entropy-large} if $\pi_*\mu=\MME$ is the measure of maximal entropy of $f$.
\end{enumerate}
\end{definition}

 Fix an ergodic $f$-invariant measure $\nu$, and denote
 \[
    L(\nu)=\{\mu:g_{*}\mu=\mu, \pi_{*}\mu=\nu\}.
\]
\begin{proposition}
    $L(\nu)$ is a nonempty convex set, which contains at least one ergodic measure.
\end{proposition}

\begin{proof}
    Convexity is clear.

    The map $\pi$ is a continuous surjection between compact metric spaces and therefore has a Borel section $s$. Define $\mu_0=s_*\nu$. For every $n\geq 1$, $\frac{1}{n}\sum_{k=0}^{n-1}g_{*}^k\mu_0$ projects to $\nu$ under $\pi$, and therefore any accumulation point $\mu$ of these averages is $g$-invariant and projects to $\nu$ as well. If 
    \[
    \mu=\int\mu_\omega\,d\tau(\omega)
    \]
    is the ergodic decomposition of $\mu$ (see \cite[Theorem 6.10 and p.~153]{Walters1982}), then
    \[
    \nu=\pi_*\mu=\int\pi_*\mu_\omega\,d\tau(\omega).
    \]
    Since $\nu$ is ergodic and hence an extreme invariant measure, we have $\pi_*\mu_\omega=\nu$ for $\tau$-almost every $\omega$.
\end{proof}

We now establish that large measures are necessarily hyperbolic. We start with a simple remark.

\begin{lemma}\label{lemma.medidas.invariantes}
Let $(g,Y,\mu)$ and $(f,X,\nu)$ be measure-preserving systems such that $\nu$ is ergodic, and let $\pi:(g,Y,\mu)\to (f,X,\nu)$ be a metric homomorphism. Then, for every $g$-invariant set $I$ such that $\mu(I)>0$, we have $\pi_*\mu_I=\nu$.
\end{lemma}
\begin{proof}
 Set $t=\mu(I)$. If $t=1$, there is nothing to prove. Otherwise,
\[
\mu=t\mu_I+(1-t)\mu_{I^c}.
\]
Since $I$ is invariant, both $\mu_I$ and $\mu_{I^c}$ are invariant.
Therefore,
\[
\nu=\pi_*\mu=t\,\pi_*\mu_I+(1-t)\,\pi_*\mu_{I^c}.
\]
Both measures on the right are $f$-invariant. Since $\nu$ is ergodic, it is an extreme point of the convex set of invariant
probability measures. Therefore, $\pi_*\mu_I=\nu$.
\end{proof}

\begin{proposition}\label{prop:nuh}
    Let $g$ be a $C^r$ diffeomorphism of a surface $M$, $r\geq 1$, preserving a measure $\mu$, and let $f$ be a homeomorphism on a set $X$ such that $g$ is semiconjugate to $f$ via a continuous map $\pi:M\to X$.
    
    If $\pi_*\mu$ is ergodic and $h_{\pi_*\mu}(f)>0$, then $\mu$ is hyperbolic.  
\end{proposition}
\bp 
Let $I^{+}=\{x:\lam^+(x;g)\leq 0\}, I^-=\{x:\lam^-(x;g)\geq 0\}$. Suppose that $\mu(I^+)>0$: since $I^+$ is invariant, $\mu_{I^+}$ is invariant. On $I^+$ both Lyapunov exponents are nonpositive. Applying  Ruelle's inequality \cite{Ruelle1978}, we obtain
\[
  0\geq h_{\mu_{I^+}}(g)\geq h_{\pi_*\mu_{I^+}}(f)=h_{\pi_*\mu}(f)>0,  
\]
which is a contradiction. An analogous argument applied to $g^{-1}$ yields $\mu(I^-)=0$, so $\mu$ is hyperbolic. 
\epr

We deduce that if $\pi:M\to M$ is a semiconjugacy between a diffeomorphism $g$ and a pseudo-Anosov $f$, then large measures are necessarily hyperbolic.\par

\medskip

\section{The Bernoulli property}
\label{sec:Bernoulli}

Let $g$ be a $C^\infty$ diffeomorphism as in the hypotheses of \Cref{mainthm:isomorphism.mu}, and let $\pi$ be the corresponding semiconjugacy. If $\mu$ is a $g$-invariant measure such that $$\pi_*\mu=\mu_{\mathrm{MME}},$$
then by \Cref{prop:nuh} $\mu$ is a hyperbolic measure. 
 In this section, we prove that $\mu$ is a Bernoulli measure. Recall that an invertible measure-preserving transformation of a probability space is {\em Bernoulli} if it is metrically isomorphic to a Bernoulli shift.\par
\smallskip
The strategy to prove the Bernoulli property in the proof of Theorem \ref{mainthm:isomorphism.mu} is to demonstrate that there exists a hyperbolic periodic point $p$ such that for every ergodic entropy-large measure $\mu$, 
\begin{equation}\label{eq.ehc.M}
    \mu(\ehc^*(p))=1.
\end{equation}
Then, $\mu(\ehc(p))=1$ for every entropy-large measure $\mu$, and by definition, all ergodic entropy-large measures are homoclinically related. By Theorem \ref{thm:BCS} (1), there is a unique entropy-large measure.
Also, $\ehc^*(p)\zeroeq\ehc(p)$, and by Remark \ref{rmk:bernoulli.mme}, $(g,\mu)$ is Bernoulli.\par
\smallskip
To prove the Bernoulli property of $m$ in  \Cref{rem:conservative} it is enough to show that there exists a hyperbolic periodic point $p$ such that 
$$m(\ehc^*(p))=1.$$
Then $\ehc^*(p)\zeroeq \ehc(p)\zeroeq M$, and by Remark \ref{rmk:bernoulli}, $(g,m)$ is Bernoulli. \par
\smallskip

We first show that,
for $\mu$-almost every $x$, the lifts of the stable and unstable
manifolds of $x$ have infinite diameter in the universal cover.
We begin by establishing that $\pi$ does not collapse intervals
inside invariant manifolds; see Section \ref{subsection.collapsible}.  A recurrence argument
then gives the infinite-diameter property; see Section \ref{sec:surjective.invariant}.\par 

Proposition \ref{pro:top.transv.int.general} implies that any $g$-stable manifold and
$g$-unstable manifold whose $\pi$-images have lifts of infinite diameter intersect
in a topologically transverse way. This statement does not depend
on the measure. Consequently, for $\mu\times \mu$-almost every pair
$(x,y)$, the $g$-stable manifold of $x$ intersects the
$g$-unstable manifold of $y$ in a topologically transverse way.\par

The diameters of the lifts of the $\pi$-images of the $g$-invariant manifolds of a periodic point $p$ are infinite in the universal cover if $\mu(\ehc(p))>0$. This holds for any ergodic large measure $\mu$. See Proposition \ref{prop:infinite.diam.periodic}.\par
To prove that there is a unique measure $\mu$ that is entropy-large, and also that it has the Bernoulli property, the strategy is to find a hyperbolic periodic point such that 
$\mu(\ehc^*(p))=1$ for all entropy-large measures $\mu$. This implies that all ergodic entropy-large measures are homoclinically related (recall  \eqref{def:homoc.related.measures}), and \Cref{thm:BCS} (1) then implies that there is a unique entropy-large measure. \Cref{rmk:bernoulli.mme} implies that $\mu$ is Bernoulli. \par

Let $\mu_0$ be any ergodic entropy-large measure. Since $\mu_0$ is hyperbolic, Katok's closing lemma implies there exists a hyperbolic periodic point $p$ such that $\mu_0(\ehc(p))=1$. The argument in Proposition \ref{prop:infinite.diam.periodic} implies that the $\pi$-images of the $g$-stable and unstable manifolds of $p$ have lifts with infinite diameter. Hence, for any ergodic entropy-large measure $\mu$, $\mu$-almost every $g$-stable manifold intersects the $g$-unstable manifold of $p$ and $\mu$-almost every $g$-unstable manifold intersects the $g$-stable manifold of $p$ in a topologically transverse way.\par
Every ergodic entropy-large measure $\mu$ as in the hypotheses of Theorem \ref{mainthm:isomorphism.mu} is a measure that maximizes the entropy of $g$, since 
\begin{equation}\label{eq:entropy.large}
h_{top}(f)=h_{top}(g)\geq h_\mu(g)\geq h_{\pi_*\mu}(f)=h_{top}(f).
\end{equation}

To prove that $\mu$-almost every invariant manifold transversely intersects the corresponding invariant manifold of $p$, we apply \Cref{thm:sard.BCS}. That is, if $\cT(p)$ is the set of points $y$ such that the $g$-stable manifold of $y$ has a non-transverse intersection with the $g$-unstable manifold of the periodic point $p$, then this exceptional family has zero transverse Hausdorff dimension if $g$ is $C^\infty$. Since a positive-entropy hyperbolic measure has positive transverse dimension, and $\mu$ is ergodic, $\cT$ has $\mu$-measure zero. In an analogous way, we obtain that $\mu$-almost every $g$-unstable manifold transversely intersects the $g$-stable manifold of $p$, whence $\mu(\ehc^*(p))=1$. The Bernoulli property follows.\par
 \smallskip

In the case of \Cref{rem:conservative}, we apply the Sard argument of \cite{HHTUSRB}. Given a
periodic point $p$ whose lifted invariant manifolds have infinite
diameter, the stable manifolds in a Pesin set form a
$C^1$-lamination. Sard's theorem implies that the union of the
stable leaves having a non-transverse intersection with
$W^u(p;g)$ has zero $m$-measure. Hence the set of points whose $g$-stable manifold transversely intersects the unstable manifold of a periodic point $p$ as in the previous paragraph is of full $m$-measure. We proceed analogously with the set of $g$-unstable manifolds intersecting the $g$-stable manifold of $p$. Hence $m(\ehc^*(p))=1$. This proves that $m$ is Bernoulli in the hypotheses of \Cref{rem:conservative}.\par\medskip

\subsection{Non-collapsible local invariant manifolds}\label{subsection.collapsible} 
Let $\mathcal{C}$ be the set of points with collapsible local unstable manifolds. That is, $x\in\cC$ if there exists $\eps>0$ such that 
\[
\pi(W^+_\eps(x;g))=\pi(x).    
\]
The set $\cC$ is $g$-invariant.

Let $\mu$ be a large measure for $g$ and set $\nu=\pi_*\mu$.

\begin{lemma}\label{claim.non.collapsible.general}
    $\mu(\cC)=0$.
\end{lemma}

\begin{proof}
    Suppose that $\mu(\cC)>0$. By \Cref{lemma.medidas.invariantes}, $\mu_{\cC}$ also projects onto $\nu$, and so do its ergodic components. Choose an ergodic component $\eta$, so $\eta(\cC)>0$: $\eta$ is an ergodic hyperbolic measure satisfying $\pi_*\eta=\nu$.  Katok's closing lemma implies the existence of a hyperbolic periodic point $p$ such that $\eta(\ehc(p))=1$. \par

    Write $\cC=\bigcup_{n>0} C_n$, where $C_n=\left\{x:\pi\left(W^+_{\frac{1}n}(x;g)\right)=\pi(x)\right\}$, and choose $n$ so that $\eta(C_n)>0$. By Poincar\'e's recurrence theorem, there exists $R_n\subset C_n$ with $\eta(C_n\setminus R_n)=0$ such that every point in $R_n$ returns infinitely many times in the past to $R_n$.  If $x\in R_n$ and $y\in W^{+}(x;g)\pitchfork W^{s}(\cO(p);g)$, then for some large iterate $k$, $g^{-k}(x)\in R_n$ and $g^{-k}(y)\in W^+_{\frac1n}(g^{-k}(x);g)$, thus $\pi(g^{-k}(x))=\pi(g^{-k}(y))$, and therefore $\pi(x)\in W^s(\cO(\pi(p));f)$. If $B=R_n\cap \ehc(p)$, then $\eta(B)>0$ and $\pi(B) \subset W^s(\cO(\pi(p));f)$. Therefore, 
    \[
    0<\eta(B)\leq \eta\left(\pi^{-1}(\pi(B))\right)=\nu(\pi(B))\quad\Rightarrow \quad\nu(W^s(\cO(\pi(p));f))>0. 
    \]
    On the other hand, $W^s(\cO(\pi(p));f)$ is an $f$-invariant set, and since $\nu$ is ergodic, it necessarily has full $\nu$-measure. Every point in 
    $W^s(\cO(\pi(p));f)$ converges to $\cO(\pi(p))$, hence by  Poincar\'e's recurrence theorem we deduce that $\nu$ is supported on $\cO(\pi(p))$. This is absurd, since we have assumed that $h_{\nu}(f)>0$.  
 \end{proof}

Applying the previous claim to $g$ and $g^{-1}$ we obtain:

\begin{proposition}\label{prop.no.colapsable}
    For $\mu$-almost every $x\in M$, $\pi(W_{\mathrm{loc}}^\pm(x;g))$ contains a nontrivial arc in $W^{u/s}(\pi(x);f)$.
\end{proposition}

\subsection{Unbounded diameter of invariant manifolds}
\label{sec:surjective.invariant}

Let $\mu$ be a large measure. Consider lifts $\widetilde f,\widetilde g,\widetilde \pi$ as given in Proposition \ref{prop.semiconj}. For a lift $\widetilde x$ of $x$, we denote by $\widetilde W^\pm(\widetilde x;\tilde g)$ the lift of $W^\pm(x;g)$ that contains $\widetilde x$. For a set contained in an unstable or stable leaf, we denote its
{\em intrinsic diameter} by $\diam_u$ or $\diam_s$, respectively.

\begin{corollary}\label{cor:infinite.diam}
For $\mu$-almost every $x$ and every lift $\widetilde{x}$ of $x$, one has
\[
\diam_u\left(\pi(W^+(x;g))\right)
=
\diam_u\left(
\widetilde{\pi}
\bigl(\widetilde{W}^+(\widetilde{x};\widetilde{g})\bigr)
\right)
=
+\infty
\]
and
\[
\diam_s\left(\pi(W^-(x;g))\right)
=
\diam_s\left(
\widetilde{\pi}
\bigl(\widetilde{W}^-(\widetilde{x};\widetilde{g})\bigr)
\right)
=
+\infty.
\]
As a consequence,
\[
\diam\left(
\widetilde{W}^{\pm}(\widetilde{x};\widetilde{g})
\right)
=
+\infty.
\]
\end{corollary}

\begin{proof}
We consider only unstable manifolds; the stable manifold case is analogous. For every $n\geq1$, let
\begin{align*}
&D_n^+=\left\{x:\diam_u\left(\pi(W^+(x;g))\right)\geq\frac1n\right\},
\intertext{and denote}
&D^+=\bigcup_{n\geq1}D_n^+.
\end{align*}
Since $\pi(W^+(x;g))$ is nontrivial for $\mu$-almost every $x$, we have $\mu(D^+)=1$. For each $n$, Poincar\'e recurrence applied to $g^{-1}$ gives a set
$R_n^+\subset D_n^+$ such that
\[
\mu(D_n^+\setminus R_n^+)=0
\]
and, for every $x\in R_n^+$, there exists a sequence $k_j\to\infty$ such that $g^{-k_j}(x)\in D_n^+$. Therefore, the set $R^+:=\bigcup_{n\geq1}R_n^+$
has full $\mu$-measure.

Fix $x\in R^+$ and choose $n$ such that $x\in R_n^+$. The invariance of unstable manifolds and the semiconjugacy relation give
\[
\pi(W^+(x;g))=f^{k_j}\left(\pi(W^+(g^{-k_j}(x);g))\right).
\]
Since $f$ expands intrinsic distances along unstable leaves by its dilatation factor $\lambda_0>1$, we obtain
\[
\begin{aligned}
\diam_u\left(\pi(W^+(x;g))\right)=\lambda_0^{k_j}\diam_u\left(\pi(W^+(g^{-k_j}(x);g))\right)\geq\frac{\lambda_0^{k_j}}{n}.
\end{aligned}
\]
Letting $j\to\infty$, we conclude that
\[
\diam_u\left(\pi(W^+(x;g))\right)=+\infty.
\]

The unstable foliation of a pseudo-Anosov homeomorphism has no closed leaves. Therefore, the projection from a lifted
unstable leaf to the corresponding leaf in $M$ is an isometry for the intrinsic leaf metrics. It follows that
\[
\diam_u\left(\widetilde{\pi}\bigl(\widetilde{W}^+(\widetilde{x};\widetilde{g})\bigr)\right)=+\infty.
\]

Since lifted  unstable leaves are properly embedded in $\widetilde{M}$, infinite intrinsic diameter implies infinite
ambient diameter. Hence,
\[
\diam\left(\widetilde{\pi}\bigl(\widetilde{W}^{+}(\widetilde{x};\widetilde{g})\bigr)\right)=+\infty.
\]

Finally, since $\|\widetilde\pi-\Id_{\widetilde M}\|$ is finite, we deduce that 
\[
\diam\left(\widetilde{W}^{+}(\widetilde{x};\widetilde{g})\right)=+\infty.
\]
\end{proof}

\subsection{Topological transverse intersection of invariant manifolds}\label{section.topological.intersection}
Let $\mu$ be a large measure. We will prove the following central proposition.

\begin{proposition}\label{pro:top.transv.int.general}
    Let $x,y$ be so that $\pi(W^+(x;g)), \pi(W^-(y;g))$ have infinite intrinsic diameter. Then, there exists a topologically transverse intersection between $W^+(x;g)$ and $W^-(y;g)$.
\end{proposition}

\smallskip 
 
\begin{definition}[Topological crossing]\label{def:topological.crossing}
 If $A,B$ are $1$-dimensional (possibly immersed) submanifolds, we say that they have a {\em topologically transverse intersection} if there exists $z\in A\cap B$ which is a crossing point. That is, there exists a disk $D$ centered at $z$ and an interval $I\subset A$ containing $z$ and separating $D$ into two  components, and there is an interval $J\subset B$ centered at $z$, such that the two connected components of $J\setminus \{z\}$ lie in different components of $D\setminus I$. We denote by $A\pitchfork_{\mathrm{top}}B$ the set of topologically transverse intersections between $A$ and $B$.   
\end{definition}
 
\smallskip 
 
By Corollary \ref{cor:infinite.diam} we deduce that:

\begin{corollary}
    For $\mu$-almost every $x,y$, $W^+(x;g) \pitchfork_{top} W^-(y;g)\neq \varnothing$.
\end{corollary}

\begin{proof}[Proof of \Cref{pro:top.transv.int.general}]
Fix $x,y$ so that $\pi(W^+(x;g)) \subset W^u(\pi(x);f)$ and $ \pi(W^-(y;g)) \subset W^s(\pi(y);f)$ have infinite intrinsic diameter. Then, each one of these sets contains one complete half-leaf (a copy of $[0,+\infty)$).  Write
 \[
 W^+(x;g)\setminus\{x\}=U_x^1\sqcup U_x^2
 \]
 where each $U_x^i$ is an open interval containing $x$ on its closure. Then $\pi(U_x^i)$ is an interval in $W^{u}(\pi(x);f)$ containing $\pi(x)$ on its closure, and since the image of the union contains a half-leaf, we conclude that for some $i$, $U_x=U_x^i$, $\pi(U_x)$ contains a half-leaf $H_x^+ \subset W^u(\pi(x);f)$. Likewise, there is a connected component $S_y$ of $W^-(y;g)\setminus\{y\}$ so that $\pi(S_y)$ contains a half-line $H_y^{-}$.

It will be convenient to use the distance $d_f$ on $\widetilde M$ obtained by lifting the distance determined by the stable and unstable foliations of $f$. This way, lifts of non-singular leaves are geodesic, and lifts of non-singular stable and unstable leaves intersect orthogonally.

Infinite half-leaves of pseudo-Anosov homeomorphisms are dense (see Expos\'e 9 in \cite{Fathi_2021}), therefore we can find $\omega\in H_x^{+}\cap H_y^{-}$ as far as we want from the initial points of the half-lines. Choose lifts $\widetilde H_x^+$ of $H_x^+$, $\widetilde H_y^-$ of $H_y^-$, and $\widetilde w$ of $w$ so that $\widetilde w\in \widetilde H_x^+\cap \widetilde H_y^-$; by the product structure between lifts of non-singular leaves of the stable and unstable foliations of $f$, it follows that $\{\widetilde w\}=\widetilde H_x^+\cap \widetilde H_y^-$, and the intersection is orthogonal. The idea now is that, taking sufficiently large intervals $I^u \subset \widetilde H_x^+, I^s \subset \widetilde H_y^-$ intersecting at $\widetilde w$, we can find $J_x^+, J_y^-$ large intervals that project onto $I_x^u, I_y^s$ under $\widetilde \pi$. Since $\widetilde \pi$ is homotopic to the identity, the intersection cannot be destroyed, and $J_x^+, J_y^-$ also intersect. 
   
Let $C:=\sup_{z\in \widetilde M, t\in [0,1]} d_f(z,\widetilde \pi_t(z))$, where $\widetilde\pi_t$ is the homotopy between $\Id_{\widetilde M}$ and $\widetilde \pi$. With no loss of generality $d_f(\widetilde w,\partial \widetilde H_x^+), d_f(\widetilde w,\partial \widetilde H_y^-)>3C$. Thus, we can find compact symmetric sub-arcs $I^u=[a,b] \subset \widetilde H_x^+, I^s=[c, d] \subset \widetilde H_y^-$ centered at $\widetilde w$ so that $I^u\cap I^s=\{\widetilde w\}$, and 
\[
    d_f(\partial I^u,I^s)=d_f(\partial I^u, \widetilde w)=d_f(I^u,\partial I^s)=d_f(\widetilde w, \partial I^s)>3C
 \]
Consider a lift $\widetilde U_x$ of $U_x$ so that $\widetilde \pi(\widetilde U_x)\supset\widetilde H_x^+$ and a lift $\widetilde U_y$ of $U_y$ so that $\widetilde \pi(\widetilde U_y)\supset\widetilde H_y^-$. Note that $\widetilde \pi|:\widetilde U_x\to \widetilde W^u(\widetilde \pi(\widetilde x);\widetilde f)$ is a continuous map between intervals. Therefore, we can find $J^+=[\widetilde a, \widetilde b] \subset \widetilde U_x$ so that $\widetilde J^+=I^u$, $\widetilde \pi(\widetilde a)=a, \widetilde\pi(\widetilde b)=b$. Similarly, we can find $J^-=[\widetilde c,\widetilde d] \subset\widetilde U_y$ that projects onto $I^s$ under $\widetilde\pi$, with boundary points going to boundary points. We want to show that $J^+\cap J^-\neq \varnothing$.

\textbf{Claim:} for every $t\in [0,1]$, $\widetilde \pi_t(\partial J^+)\cap \widetilde\pi_t (J^-)=\varnothing$ and $\widetilde \pi_t(J^+)\cap \widetilde\pi_t (\partial J^-)=\varnothing$.

\smallskip 
 
Indeed, take $z\in J^-$ and compute
\begin{align*}
d_f(\widetilde \pi_t(\widetilde a), \widetilde\pi_t(z))&\geq d_f(a,\widetilde \pi(z))-2C\geq 3C-2C=C>0.
 \end{align*} 
Similarly for the other cases. 

Consider arc length parametrizations $\alpha^\pm$ of $J^{\pm}$, and denote $\alpha_t^{\pm}=\widetilde\pi_t(\alpha^{\pm})$. Then, 
$\operatorname{Im}(\alpha_0^{\pm})=J^\pm$, while for $t=1$ we get $\operatorname{Im}(\alpha_1^{+})=I^u, \operatorname{Im}(\alpha_1^{-})=I^s$. Let $\beta_t^{\pm}$ be defined on $[0,2]$ so it coincides with $\alpha_t^{\pm}$ for $t\in [0,1]$, and for $t\in [1,2]$ is a homotopy inside $\widetilde\pi(J^{\pm})$ relative to its endpoints, such that $\beta_2^{\pm}$ parametrizes $\widetilde\pi(J^{\pm})$ by arc length.

If $\iota_2$ denotes the $\bmod 2$ intersection number between curves, it follows that 
\[
    \iota_2(\beta^+_0,\beta^-_0)=\iota_2(\beta^+_2,\beta^-_2)=\#(I^u\cap I^s)=1.
\]
This implies that necessarily $\operatorname{Im}(\beta^+_0)\cap \operatorname{Im}(\beta^-_0)=J^{+}\cap J^-\neq \varnothing$. Necessarily, there is crossing intersection, because otherwise by making a small homotopy we could make $J^{+}\cap J^-= \varnothing$, contradicting the fact that the intersection number is non-zero.

From here we deduce that $W^+(x;g)$ intersects $W^-(y;g)$ in a topologically transverse way.
\end{proof}

\begin{remark}
Fix $\sigma\in\{+,-\}$, a point $x$, and a lift $\widetilde{x}$ of
$x$. If
\begin{align}\label{eq:infinitediamter}
\diam\left(\widetilde{W}^{\sigma}(\widetilde{x};\widetilde{g})\right)=+\infty,
\end{align}
then, for $\mu$-almost every $y$,
\[
W^{\sigma}(x;g)
\pitchfork_{\mathrm{top}}
W^{-\sigma}(y;g)
\neq\varnothing.
\]
Indeed, since $\widetilde \pi$ is at bounded distance to the identity, equality \eqref{eq:infinitediamter} implies that 
\[
 \diam\left(\widetilde\pi(\widetilde{W}^{\sigma}(\widetilde{x};\widetilde{g}))\right)=+\infty,   
\]
and therefore its intrinsic diameter is also infinite. On the other hand, Corollary \ref{cor:infinite.diam} provides the same property for the opposite invariant manifold. This was what we used in the previous proof.

Equivalently, for $\mu$-almost every $y$, there exists a lift $\widetilde{y}$ of $y$ such that
\[
\widetilde{W}^{\sigma}(\widetilde{x};\widetilde{g})
\pitchfork_{\mathrm{top}}
\widetilde{W}^{-\sigma}(\widetilde{y};\widetilde{g})
\neq\varnothing.
\]
\end{remark}

\begin{remark}\label{rem:finite.intersection}
The proof of \Cref{pro:top.transv.int.general} gives the following more precise statement.
Let $C$ be as in that proof. Suppose that
\[
I^u\subset \widetilde W^u(\widetilde\pi(\widetilde x);\widetilde f),
\qquad
I^s\subset \widetilde W^s(\widetilde\pi(\widetilde y);\widetilde f)
\]
are compact intervals such that $I^u\cap I^s=\{\widetilde w\}$ and
\[
d_f(\partial I^u,I^s)>3C,
\qquad
d_f(I^u,\partial I^s)>3C.
\]
If
\[
J^+\subset\widetilde W^+(\widetilde x;\widetilde g),
\qquad
J^-\subset\widetilde W^-(\widetilde y;\widetilde g)
\]
are compact intervals satisfying
\[
\widetilde\pi(J^+)=I^u,
\qquad
\widetilde\pi(J^-)=I^s,
\]
then
\[
J^+\pitchfork_{\mathrm{top}}J^-\neq\varnothing.
\]
In particular, if $R_+(x)=\infty$ and $R_-(y)=\infty$ and
$$\widetilde W^u(\widetilde\pi(\widetilde x);\widetilde f)\cap \widetilde W^s(\widetilde\pi(\widetilde y);\widetilde f)\ne\varnothing,$$
then
$$\widetilde W^+(\widetilde x;\widetilde g)\transv_{top} \widetilde W^-(\widetilde y;\widetilde g)\ne\varnothing.$$
\end{remark}

\subsection{\textbf{Non-trivial invariant manifolds of periodic points}} Next we consider periodic points.

\begin{lemma}
    Let $p$ be a hyperbolic periodic point and $\mu$ a hyperbolic measure. Then 
    \[
    \mu(\ehc^*(p)\setminus \overline{W^{\pm}(p;g)})=0.
    \]
\end{lemma}

\begin{proof}
Let $x\in\ehc^*(p)$ be a bi-recurrent point for $g^{\peri(p)}$, and let $\eps>0$. There exists $z\in W^-(x;g)\transv W^+(p;g)$. Then there exists $N>0$ such that $d(g^{n\peri(p)}(x),g^{n\peri(p)}(z))<\eps/2$ for all $n\geq N$. Since $x$ is recurrent, there exists $n\geq N$ such that $d(x,g^{n\peri(p)}(x))<\eps/2$. Since $\eps>0$ is arbitrary, the triangle inequality implies that $d(x,W^+(p;g))=0$ and hence $x\in\overline{W^+(p;g)}$. In an analogous way, we prove $x\in\overline{W^-(p;g)}$. Since the set of bi-recurrent points has full $\mu$-measure, the
conclusion follows.
\end{proof}

\begin{proposition}\label{prop:infinite.diam.periodic}
    Let $\mu$ be a large measure. If $\mu(\ehc^*(p))>0$ and $\widetilde p$ is a lift of $p$, then  $\diam(\widetilde W^\pm(\widetilde p;\widetilde g))=\infty$.
\end{proposition}

\begin{proof}
Let $\nu=\pi_*\mu$. We prove the statement for the unstable manifold. The stable case is analogous. By the previous lemma, $\mu\left(\ehc^*(p)\setminus\overline{W^+(p;g)}
\right)=0$, hence we have
\[
\nu\left(\pi\left(\overline{W^+(p;g)}\right)\right)=\mu\left(\pi^{-1}\left(\pi\left(\overline{W^+(p;g)}\right)\right)\right)
\geq\mu(\ehc^*(p))>0.
\]

Since $\nu$ is ergodic and has positive entropy, it is non-atomic, and therefore $\pi\left(\overline{W^+(p;g)}\right)$ cannot be a point.  Therefore,
$\pi(W^+(p;g))$ is not a point. 

Let $\ell$ be the period of $p$, $F=f^\ell, G=g^\ell$. Then $G(W^+(p;g))=W^+(p;g)$, and thus $F\left(\pi(W^+(p;g))\right)=\pi(W^+(p;g))$. Since $F$ expands intrinsic distances along unstable leaves by $\lambda_0^\ell>1$, we obtain
\[
\diam_u\left(\pi(W^+(p;g))\right)=\diam_u\left(F\left(\pi(W^+(p;g))\right)\right)=\lambda_0^\ell\diam_u\left(\pi(W^+(p;g))\right).
\]
The diameter is nonzero, and $\lambda_0^\ell>1$. It follows that
\[
\diam_u\left(\pi(W^+(p;g))\right)=+\infty.
\]

As in the final part of the proof of Corollary \ref{cor:infinite.diam}, this implies
\[
\diam_u\left(
\widetilde{\pi}
\bigl(\widetilde{W}^+(\widetilde{p};\widetilde{g})\bigr)
\right)
=
+\infty.
\]
Properness of the lifted unstable leaves and the fact that
$\widetilde{\pi}$ is at bounded distance from the identity then give
\[
\diam\left(
\widetilde{W}^+(\widetilde{p};\widetilde{g})
\right)
=
+\infty.
\]
\end{proof}

As a consequence, for $\mu$-almost every $x$, 
\[
    W^{+}(x;g)\pitchfork_{\mathrm{top}} W^-(p;g)\neq\varnothing\quad \text{ and }\quad W^{-}(x;g)\pitchfork_{\mathrm{top}} W^+(p;g)\neq\varnothing.
\]

\subsection{The Bernoulli property}\label{sec:Sard}

\begin{proposition}\label{prop:Bernoulli}
Let $\mu$ be an entropy-large measure in the hypotheses of  \ref{mainthm:isomorphism.mu}. Then the system $(g,\mu)$ is Bernoulli.\par
 Moreover, there is only one entropy-large measure $\mu$ in the hypotheses of Theorem \ref{mainthm:isomorphism.mu}.    
\end{proposition}

\begin{proof}
The heart of the proof is passing from topological transverse intersections between stable and unstable manifolds to real (differentiable) transverse intersections. Let $\mu$ be a measure such that $\pi_*\mu=\MME$, the entropy-maximizing measure of $f$. Then, as discussed in \eqref{eq:entropy.large}, $\mu$ is a measure of maximal entropy for $g$. We can then use \Cref{thm:sard.BCS}.  \par

As mentioned before, the strategy is to prove that there exists a periodic point $p$ such that $\mu$-almost every $x$ belongs to $\ehc^*(p)$. Since $\mu$ is hyperbolic there exists a hyperbolic periodic point $p$ so that $\mu(\ehc^*(p))>0$. Then, by Propositions \ref{prop:infinite.diam.periodic} and \ref{pro:top.transv.int.general}, for every large measure $\mu'$ and for $\mu'$-almost every $x$, $W^-(x;g)$ intersects $W^+(p;g)$ in a topologically transverse way. 
We will show that, for $\mu$-almost every $x$, $W^-(x;g)$ transversely intersects $W^+(p;g)$. An analogous argument will show that $W^+(x;g)$ transversely intersects $W^-(p;g)$, and hence $x\in\ehc^*(p)$.\par

Let $P$ be a Pesin block for $\mu$. Then for $\mu$-almost every $x\in P$, there exists a small rectangle $R=R_x$ such that $\partial R$ is formed by pieces of transverse stable and unstable leaves, $\mu(\partial R)=0$, and $x\in \inte R$. 
For $\mu$-almost every $y\in P\cap R$, $y\mapsto W^-(y)\cap R$ is a $\cC^1$-lamination of $R$. 
\newline
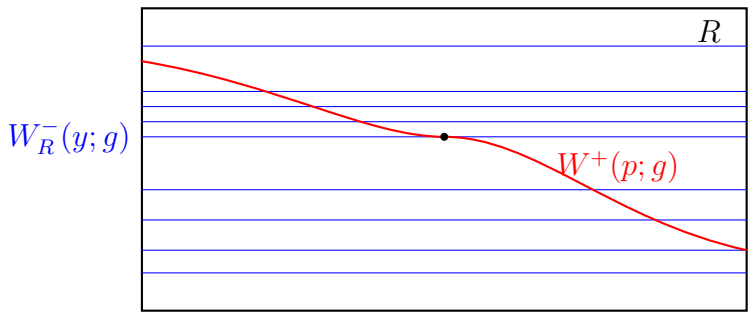
\par

\begin{tikzpicture}[scale=1.0]
\begin{centering}
\draw[thick] (0,0) rectangle (8,4);
\node at (7.5,3.7) {$R$};

\foreach \y in {0.5,0.8,1.2,1.6,2.3,2.5,2.7,2.9,3.5} {
    \draw[blue] (0,{\y}) -- (8,{\y});
}
\node[left,blue] at (0,2.3) {$W^-_R(y;g)$};

\draw[red,thick]
    (0,3.3)
    .. controls (1.8,3.0) and (3.0,2.3) ..
    (4,2.3)
    .. controls (5.2,2.3) and (6.2,1.2) ..
    (8,0.8);

\fill (4,2.3) circle (1.5pt);

\node[red] at (6.3,1.9) {$W^+(p;g)$};
\end{centering}
\end{tikzpicture}

\smallskip

Now, for $\mu$-almost every $x\in P$, $x$ is a recurrent point for $G=g^\ell$, where $\ell$ is the period of $p$, and $W^-(x;g)$ has a topologically transverse intersection with $W^+(p;g)$, as proven in the previous part.

Since $d(G^{k}(x),W^+(p;g))\xrightarrow[k\mapsto\oo]{} 0$, and $x$ is $G$-recurrent, $W^+(p;g)\cap \inte(R)\ne\varnothing$ and $W^+(p;g)$ topologically crosses the stable plaques in $R$. Let $W^-_R(y;g)$ be the connected component of $W^-(y;g)\cap R$ containing $y$. Define
$$
{\mathcal T}=\{y\in R: W^-_{R}(y)\text{ has a non-transverse intersection with }W^+(p)\}
$$

The adapted Sard's lemma \Cref{thm:sard.BCS} implies that the exceptional family $\mathcal{T}$ has zero transverse Hausdorff dimension. Since a positive-entropy hyperbolic measure has positive transverse dimension \cite{LedYoungII}, $\mathcal T$ cannot carry all the $\mu$-mass of $R$. Hence, in both cases, we obtain:
\begin{equation}\label{eq.transversal}
    \mu(\{y\in R: W^-_R(y;g)\pitchfork W^+(p;g)\ne\varnothing\})>0.
\end{equation}

Let $A^-_R(p)=\{y\in R: W^-_R(y;g)\pitchfork W^+(p;g)\ne\varnothing\}$. In an analogous way, we define $A^+_R(p)$. Since $W^+(p;g)$ has infinite intrinsic diameter, we have the following situation:
\begin{enumerate}
    \item for every ergodic entropy-large measure $\mu'$ that is a measure of maximal entropy there exists a  suitable rectangle $R=R(\mu')$ such that $\mu'(A^-_R(p))>0$.  
\end{enumerate}
\par
Consider the ergodic decomposition $\{\mu'\}$ of $\mu$ for $g$, and let $\mu'$ be an ergodic component of $\mu$. The measure $\mu'$ decomposes into finitely many $G$-ergodic components $\mu'_1,\dots,\mu'_q$. 
\smallskip
If $\mu$ is in the hypotheses of Theorem \ref{mainthm:isomorphism.mu}, we will show that all its $g$-ergodic components are homoclinically related. Hence, by \Cref{thm:BCS} (1), there is a unique ergodic component of $\mu$. Therefore, $\mu$ is ergodic and it is the unique entropy-large measure and it satisfies $\mu(\ehc^*(p))=1$.\par

Let $\mu'$ be an ergodic component of $\mu$. Then $\mu'$ is an
ergodic entropy-large measure that maximizes entropy. Let $\mu_i'$
be a $G$-ergodic component of $\mu'$. Applying the preceding argument
to $G=g^\ell$ and $F=f^\ell$, there exists a suitable rectangle $R$
such that
\[
\mu_i'(A_R^\mp(p))>0.
\]
Since the sets
$$\ehc^*_\pm(p)=\{y: W^\pm(y)\transv W^\mp(p)\ne\varnothing\}$$
are $G$-invariant, $A^\pm_R(p)\subset\ehc^*_\pm(p)$, and $\mu'_i$ is $G$-ergodic, we have that $\mu'_i(\ehc^*(p))=1$ for all $i=1,\dots,q$, and hence $\mu'(\ehc^*(p))=1$. By our definition of homoclinically related measures \eqref{def:homoc.related.measures}, all ergodic components $\mu'$ of $\mu$ are homoclinically related. Item (1) of Theorem \ref{thm:BCS} implies that there is only one entropy-large measure that maximizes entropy. 
\par
As explained at the beginning of the section, this implies that $(g,\mu)$ is Bernoulli.
\end{proof}

If $\mu=m$ is the area measure, then the Sard argument introduced in \cite{HHTUSRB} (\Cref{thm:sard.HHTU}) implies that the $m$-measure of the set ${\mathcal T}$ in the proof above is zero. 
Defining $A^\pm_R(p)$ as above, for every $\mu'\ll m$ there exists a suitable rectangle $R=R(\mu')$ such that $\mu'(A_R^\pm(p))>0$.
The area measure has countably many ergodic components,  see \cite{pesin76} and Theorem \ref{teo.ergodic.spectral}. This implies that each $g$-ergodic component satisfies $\mu'\ll m$. Therefore, each $G$-ergodic component satisfies $\mu'_i\ll  m$, and hence, by the remark above, $\mu'_i(A^\pm_R(p))>0$, for some suitable rectangle $R$. Since $A^\pm_R(p)\subset \ehc^*_\pm(p)$ and this set is $G$-invariant, then $\mu'_i(\ehc^*(p))=1$ for all $i$, and hence $\mu'(\ehc^*(p))=1$ for all $\mu'$. Hence $m(\ehc^*(p))=1$. As explained at the beginning of the Section, this implies that $m$ is Bernoulli.\par

\section{\texorpdfstring{$\pi$}{pi} is an isomorphism}
\label{sec:isomorphism}

We fix an ergodic large measure $\mu$ for $g$; in particular $\mu$ is hyperbolic. For what follows, we rely on the theory of entropy and measurable partitions, as appears, for example, in the classical treatise of Rokhlin \cite{Rokhlin_1967}. 

If $\xi$ is a measurable partition of the probability space $(M,\mu)$, we denote by $\xi(x)$ the element of $\xi$ containing $x$, and $\{\mu_x^{\xi}\}$ a corresponding system of conditional measures. 

\begin{definition}
    $\xi$ is increasing for $g$ if $g^{-1}\xi\geq \xi$. 
\end{definition}
The partition $g^{-1}\xi$ has elements $\{g^{-1}\left(\xi(g(x))\right)\}$; it is also measurable. If $\xi$ is increasing for $g$, then 
\[
h_{\mu}(g,\xi)=H_{\mu}(g^{-1}\xi|\xi)=-\int \log \mu^{\xi}_x((g^{-1}\xi)(x)) d\mu(x).
\]

Since $\mu$ is hyperbolic, it is possible to find an increasing partition $\xi$ for $g$ which additionally verifies, for $\mu-$almost every $x$,
\begin{enumerate}
\item $\xi(x)$ is subordinated to $W^+$,
\item $\xi(x)$ contains a relatively open neighborhood of $x$ (an interval),
\item $g^{-n}\xi \xrightarrow[n\to\infty]{}$ Borel $\sigma-$algebra.
\end{enumerate}

See \cite{LedStrelcyn}. In this case, it follows that $h_{\mu}(g)=h_{\mu}(g,\xi)$. Denote by $\mathcal{P}_{\pi}$ the (measurable) partition consisting of pre-images of $\pi$: $\mathcal{P}_\pi(x)=\pi^{-1}(\pi( x))$. Note that $g\mathcal P_\pi=\mathcal P_\pi$.

\smallskip 
 
\textbf{Notation:} We write $\mathcal{Z}=\{\eta \text{ partition of finite entropy}\}$. If $\eta\in\mathcal{Z}$ we denote
\begin{align*}
h_{\mu}(g|\mathcal{P}_\pi)=\sup_{\eta\in\mathcal{Z}}H_{\mu}\left(\eta|\bigvee_{n=1}^{\infty}g^{-n}\eta\vee \mathcal{P}_\pi\right).
\end{align*}
If $\eta_1,\eta_2$ are $\sigma$-algebras, we write $\eta_1 \subset_{\mu} \eta_2$ if for every $A\in \eta_1$ there exists $B\in\eta_2$ so that $\mu(B\triangle A)=0$. If $\eta_1,\eta_2$ are partitions, $\eta_1 \subset_{\mu} \eta_2$ is understood as referring to the $\sigma$-algebras that they generate.

\begin{proposition}\label{prop:injectivity.unstables} Let $\mu$ be an ergodic large measure such that $h_\mu(g)=h_{\pi_*\mu}(f)$. Then for $\mu$-almost every $x$,
$$\mu^\zeta_x=\delta_x,$$
where $\{\mu^\zeta_x\}$ is a family of conditional measures with respect to the partition $\zeta=\xi\vee\cP_\pi$, and $\delta_x$ is the Dirac measure supported on $x$.
An analogous statement holds for $W^-(x)$.
\end{proposition}

\begin{proof} 
    We only consider the first case, the other one follows since $h_{\mu}(g)=h_{\mu}(g^{-1})$. Let $\nu=\pi_*\mu$. The Abramov-Rokhlin formula \cite{abramov1966entropy} tells us that
    \[
    h_{\mu}(g)=h_{\nu}(f)+h_{\mu}(g|\mathcal{P}_\pi),
    \]
    and therefore, $h_{\mu}(g|\mathcal{P}_\pi)=h_{\mu}(g^{-1}|\mathcal{P}_\pi)=0$. This means that for any $\eta\in\mathcal{Z}$,
    \[
    \eta \subset_{\mu} \bigvee_{n=1}^{\infty}g^{n}\eta\vee \mathcal{P}_\pi.
    \]
    The partition $\xi$ is countably generated; therefore, there exists $(\eta^k)_k \subset \mathcal{Z}$ such that $\eta^k \nearrow \xi$, and thus
    \[
    \eta^k \subset_{\mu}  \bigvee_{n=1}^{\infty}g^{n}\eta^k\vee \mathcal{P}_\pi \subset_{\mu}  \bigvee_{n=1}^{\infty}g^{n}\xi\vee \mathcal{P}_\pi
    \]
    and
    \[
    \xi  \subset_{\mu}  \bigvee_{n=1}^{\infty}g^{n}\xi\vee \mathcal{P_\pi}\]

 Set $\zeta=\xi\vee\cP_\pi$. Then, since $\xi$ is increasing and $\cP_\pi\subset_\mu \bigvee_{n=1}^{\infty}g^{n}\xi\vee \mathcal{P_\pi}$, we have
 $$\zeta=_\mu\bigvee_{n=1}^{\infty}g^{n}\xi\vee \mathcal{P_\pi}.$$
Applying $g^{-1}$, we have
$$g^{-1}\zeta=_\mu \bigvee_{n=0}^{\infty}g^{n}\xi\vee \mathcal{P_\pi}=_\mu \zeta.$$
Then $g^{-n}\zeta=_\mu\zeta$ for every $n\geq 0$. Since $\zeta$ is increasing, and by the properties of $\xi$, we have

$$\eps=_\mu\bigvee_{n\geq 0}g^{-n}\xi\subset_\mu \bigvee_{n\geq 0}g^{-n}\zeta=_\mu\zeta,$$
where $\eps$ is the partition into points. Hence $\zeta=_\mu\eps$.\par
Let $\mu^\zeta_x$ be a family of conditional measures with respect to the partition $\zeta$. Then the equation above implies that, for $\mu$-almost every $x$,
$$\mu^\zeta_x=\delta_x.$$

An analogous statement holds for the partition $\xi^-$ subordinate to $W^-$, and $\zeta^-=\xi^-\vee\cP_\pi$. This concludes the proof.
\end{proof}
\subsection{Surjectivity of \texorpdfstring{$\pi$}{pi} on invariant manifolds}
Let $\mu$ be an ergodic large measure such that $h_\mu(g)=h_{\pi_*\mu}(f)$.  
For $x\in M$ define the \emph{internal radius} of $\pi(W^\pm(x))$ by  
\begin{align}\label{eq:internalradius}
R_\pm(x)=d_{\pm}(\pi(x),\partial \pi(W^\pm(x;g))).
\end{align} 
The internal radius is defined as $\infty$ if the boundary appearing in \eqref{eq:internalradius} is empty.
\begin{lemma}\label{lem:positive.radius} 
For $\mu$-almost every $x$, $R_{\pm}(x)>0$.
\end{lemma}
\begin{proof}
Let $$Z_\pm=\{x:R_\pm(x)=0\}.$$ The sets $Z_\pm$ are invariant. Since $\mu$ is ergodic, $\mu(Z_\pm)\in\{0,1\}. $ Suppose $\mu(Z_+)=1$, the stable case is analogous.\par
Let $\xi$ be the subordinate unstable partition, and let $\zeta=\xi\vee\cP_\pi$. By Proposition \ref{prop:injectivity.unstables}, $\mu^\zeta_x=\delta_x$ $\mu$-a.e. $x$. This implies that there exists a full-measure set $X$ such that for every $x\in X$,
$$\pi|_{X\cap \xi(x)}\quad\text{is injective}.$$
Since $\mu(Z_+\cap X)=1$, disintegrating along $\xi$, we obtain that $\mu$-a.e. $x$,
\begin{equation}\label{eq:atomic}
\mu^\xi_x(Z_+\cap X)=1.    
\end{equation}

By Corollary \ref{cor:infinite.diam}, $\mu$-almost every $x$ satisfies $$\diam_u(\pi(W^+(x)))=\infty,$$
hence $\partial \pi(W^+(x))$ is either empty or a point. For every $x\in Z_+\cap X$, $R_+(x)=0$, then $\pi(x)\in\partial\pi(W^+(x)).$ For all $y\in\xi(x)\cap X\setminus\{x\}$, $\pi(W^+(y))=\pi(W^+(x))$; and, since $\pi$ is injective on $X\cap\xi(x)$, $\pi(x)\ne\pi(y)$, hence $R_+(y)>0$. This implies that $X\cap Z_+\cap\xi(x)$ contains only one point and \eqref{eq:atomic} implies that $\mu^\xi_x$ is atomic for $\mu$-almost every $x$.
This is impossible: by
\cite[Proposition~7.3.1 and Theorem~C$'$]{LedYoungII},
the assumption \(h_\mu(g)>0\) implies that the conditional measures
of \(\mu\) along unstable plaques are non-atomic. Therefore
\(\mu(Z_+)=0\), and hence
\[
R_+(x)>0
\]
for \(\mu\)-almost every \(x\). The proof for \(R_-\) is analogous.
\end{proof}
\begin{proposition}\label{prop:leafwise.surjectivity}
There exists a full-measure set $X$ such that for every $x\in X$ 
\[
R_+(x)=R_-(x)=+\infty.
\]
Consequently,
\[
\pi(W^+(x;g))=W^u(\pi(x);f)
\]
and
\[
\pi(W^-(x;g))=W^s(\pi(x);f)
\]
for every $x\in X$.

Equivalently, for every lift $\widetilde x$ of every $x\in X$,
\[
\widetilde\pi\left(\widetilde W^{\pm}(\widetilde x;\widetilde g)\right)
=
\widetilde W^{u/s}\left(\widetilde\pi(\widetilde x);\widetilde f\right).
\]
\end{proposition}

\begin{proof}
The fact that $R_{\pm}(x)=\oo$ for $\mu$-almost every point follows exactly as in the proof of \Cref{cor:infinite.diam}, using the previous lemma instead of \Cref{prop.no.colapsable}. Now $\pi(W^+(x;g))$ is a connected interval in $W^u(\pi(x);f)$. Every proper interval in a pseudo-Anosov leaf has a
boundary point at finite intrinsic distance from each of its points. Thus, $R_{\pm}(x)=+\infty$ implies
\[
\pi(W^{\pm}(x;g))=W^{u/s}(\pi(x);f).
\]
For the second part we use that stable and unstable leaves of a pseudo-Anosov homeomorphism are not closed, hence the covering projection restricts to a homeomorphism (in fact, to an isometry of the intrinsic metric), from each lifted leaf to its corresponding leaf downstairs. The downstairs
equalities therefore lift to
\[
\widetilde\pi\left(
\widetilde W^\pm(\widetilde x;\widetilde g)
\right)
=
\widetilde W^{u/s}\left(
\widetilde\pi(\widetilde x);\widetilde f
\right)
\]
for every lift of every point in the full-measure set above.
\end{proof}

\subsection{\texorpdfstring{$\pi$}{pi} is an isomorphism} Let $\mu$ be an entropy-large $g$-invariant measure.
\begin{proposition}\label{prop:injectivity.unstable}
There exists an invariant full-measure set $X\subset M$ such that for every lift $\widetilde x$ of a point $x\in X$, 
$$\widetilde\pi^{-1}(\widetilde\pi(x))\cap \widetilde W^\pm(\widetilde x;\widetilde g)=\{\widetilde x\}.$$
\end{proposition}
\begin{proof}
 We first reduce $X$, if necessary, so that Proposition \ref{prop:injectivity.unstables} holds
simultaneously at every iterate of every point of $X$. Thus, if $\xi^\pm$ are $g^{\pm 1}$-increasing partitions subordinate to $W^\pm$, then $\pi$ is injective on
$X\cap\xi^\pm(x)$ for every $x\in X$.

Since
\[
W^+(x;g)
=
\bigcup_{n\geq 0}
g^n\bigl(\xi^+(g^{-n}(x))\bigr),
\]
the invariance of $X$ implies that
\[
\left.\pi\right|_{X\cap W^+(x;g)}
\]
is injective. \par
 After reducing \(X\) and making it invariant, we may assume that Proposition \ref{prop:leafwise.surjectivity} holds for every point of \(X\), and that \(\widetilde\pi\) is injective on
\[
\widetilde X\cap\widetilde W^\pm(\widetilde x;g)
\]
for every lift \(\widetilde x\) of a point \(x\in X\). $\widetilde X$ is the lift of $X$ to the universal cover. In particular, for all $x,y\in X$ there exist lifts $\widetilde x,\widetilde y$ such that 
$$\widetilde W^+(\widetilde x;\widetilde g)\transv_{top}\widetilde W^-(\widetilde y;\widetilde g)\ne\varnothing.$$
\par
\begin{claim}\label{claim:two.sided.approximation}
After removing a set of zero measure, every point
$y\in X\cap\xi^\pm(x)$ is accumulated by
$X\cap\xi^\pm(x)$ from both sides in $\xi^\pm(x)$.
\end{claim}

\begin{proof}
Fix an atom $\xi^\pm(x)$, and let $N_x^\pm$ be the set of points in
$\operatorname{supp}\mu_x^{\xi^\pm}$ which are accumulated by this
support from only one side. Every point of $N_x^\pm$ is an endpoint
of a component of
\[
\xi^\pm(x)\setminus\operatorname{supp}\mu_x^{\xi^\pm}.
\]
Hence $N_x^\pm$ is countable. Since the conditional measure
$\mu_x^{\xi^\pm}$ is non-atomic,
\[
\mu_x^{\xi^\pm}(N_x^\pm)=0.
\]
Moreover, $\mu_x^{\xi^\pm}(X\cap\xi^\pm(x))=1$. It follows that,
after removing a null set, every point of $X\cap\xi^\pm(x)$ is
accumulated by $X\cap\xi^\pm(x)$ from both sides.
\end{proof}

\smallskip
Let $\widetilde x\in \widetilde X$ and assume there exists $\widetilde y\in \widetilde W^+(\widetilde x;\widetilde g)$ such that $\widetilde\pi(\widetilde x)=\widetilde\pi(\widetilde y)$. Call $J^+$ the unstable interval $[\widetilde x,\widetilde y]$ contained in $\widetilde W^+(\widetilde x;\widetilde g)$ having $\widetilde x,\widetilde y $ as endpoints. From \Cref{claim:two.sided.approximation} it follows that $\widetilde x$ is approximated in $J^+$ by points of $\widetilde X$. This implies that $\widetilde \pi(J^+)$ is nondegenerate. Thus, there exists an iterate of $J^+$, which we continue to call $J^+$, such that $\widetilde\pi(J^+)=I^u$ satisfies $\diam_u(I^u)>20C$. There are disjoint intervals $J^+_1,J^+_2\subset J^+$ such that $\widetilde\pi(J^+_1)=\widetilde\pi(J^+_2)$, $\widetilde\pi(\widetilde x)\in\partial\widetilde\pi(J^+_i)$, and $\diam_u(\widetilde\pi(J^+_i))>7C.$\par
By an abuse of notation, continue to call $\xi^+$ the subordinate partition to $\widetilde W^+$ and $\mu^{\xi^+}$ the family of conditional measures.
We may assume, passing to a further iterate if necessary, that there exists $w\in J^+_1\cap \widetilde X$ such that $\mu^{\xi^+}_w(\widetilde X\cap \xi^+(w))=1$ and $d_f(\widetilde\pi(w'),\partial\widetilde\pi( J^+_1))>3C$ for every $w'\in \xi^+(w)\cap \widetilde X$.  \par
By our choice of $x$ and by \Cref{rem:finite.intersection}, 
$$W^-(w';g)\transv_{top} J^+_i\ne\varnothing\quad i=1,2\quad\forall w'\in \xi^+(w)\cap \widetilde X.$$
By the argument used in the proof of \Cref{prop:Bernoulli}, there is a $\mu^{\xi^+}_w$-positive measure set $B\subset X\cap\xi^+(w)$ such that 

$$W^-(w')\transv J^+_i\ne\varnothing\quad i=1,2\quad \forall w'\in B.$$
This means that the stable holonomy map $h^-$ between a subset of $J^+_1$ and $J^+_2$ is well-defined for a positive measure set $B'\subset B\subset J^+_1$. By Ben Ovadia's result on product measures, \cite[Theorems~5.7 and~6.1]{Ovadia2023}, the unstable conditional
measures of $\mu$ form an absolutely continuous family under local stable holonomies. This means that 
$$\mu^{\xi^+}_{h^-(w)}(h^-(B'))>0.$$
For each $w'\in B'$, $$w', h^-(w')\in \widetilde W^+(\widetilde x;\widetilde g)\cap \widetilde W^-(\widetilde w';\widetilde g).$$
\Cref{prop.semiconj} implies that $$\widetilde\pi(w')=\widetilde\pi(h^-(w'))\quad\forall w'\in B'.$$
Therefore, $$h^-(B')\cap \widetilde X=\varnothing.$$
This fold-holonomy argument and the resulting contradiction are illustrated in Figure~\ref{fig:fold.holonomy}.

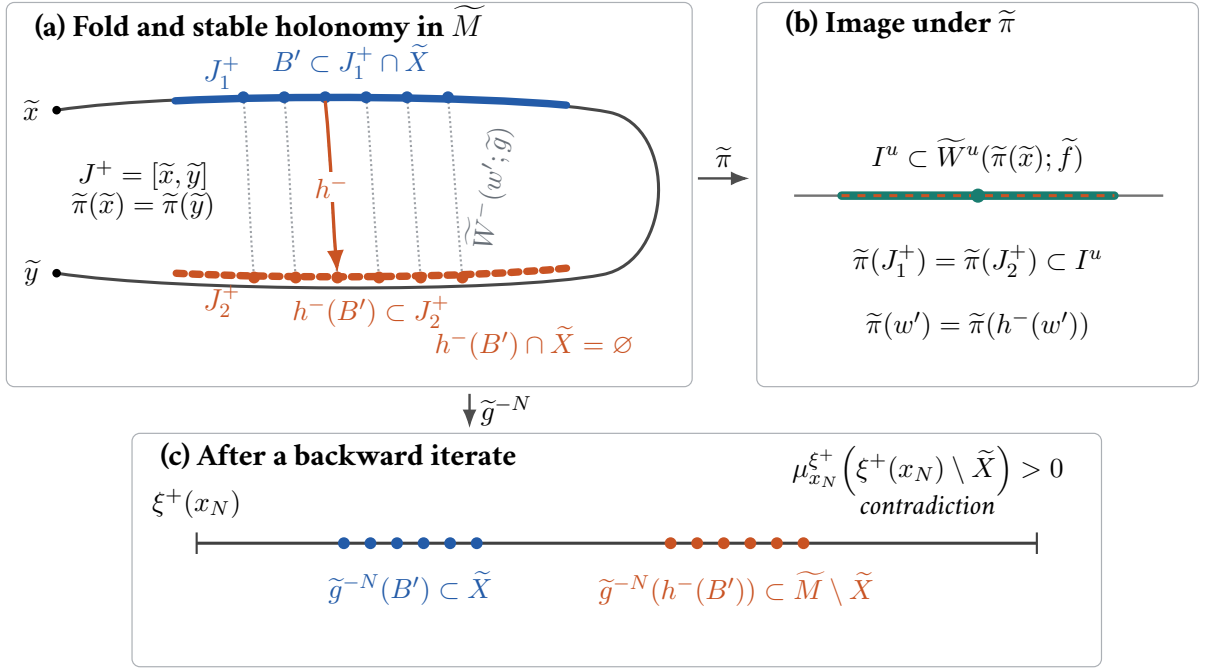
\begin{figure}   
    \centering

\begin{tikzpicture}[
    scale=.95,
    transform shape,
    x=1cm,
    y=1cm,
    >=Latex,
    font=\small,
    panel/.style={draw=softgray!55, rounded corners=3pt, line width=.45pt},
    stable/.style={draw=softgray!65, densely dotted, line width=.8pt},
    unstable/.style={draw=black!72, line width=1.15pt},
    maparrow/.style={->, line width=.9pt, draw=black!70},
    note/.style={align=center, inner sep=2pt}
]

\draw[panel] (-.25,2.80) rectangle (9.30,8.15);
\node[anchor=west,font=\bfseries] at (0.00,7.85)
    {(a) Fold and stable holonomy in $\widetilde M$};

\draw[unstable]
    (0.45,6.65)
    .. controls (2.20,6.92) and (6.85,6.88) .. (8.15,6.62)
    .. controls (9.10,6.42) and (9.05,4.55) .. (8.05,4.36)
    .. controls (6.70,4.10) and (2.15,4.12) .. (0.45,4.38);

\fill (0.45,6.65) circle (1.7pt);
\fill (0.45,4.38) circle (1.7pt);
\node[anchor=east] at (0.35,6.65) {$\widetilde x$};
\node[anchor=east] at (0.35,4.38) {$\widetilde y$};
\node[note,anchor=west] at (0.58,5.53)
    {$J^+=[\widetilde x,\widetilde y]$\\[-1pt]
     $\widetilde\pi(\widetilde x)=\widetilde\pi(\widetilde y)$};

\draw[branchblue,line width=2.8pt,line cap=round]
    (2.10,6.78) .. controls (3.65,6.86) and (6.20,6.84) .. (7.55,6.72);
\draw[branchorange,line width=2.8pt,line cap=round,dashed]
    (7.55,4.45) .. controls (6.20,4.31) and (3.65,4.30) .. (2.10,4.38);
\node[branchblue,font=\bfseries,above] at (2.75,6.80) {$J_1^+$};
\node[branchorange,font=\bfseries,below] at (2.75,4.36) {$J_2^+$};

\foreach \u/\v in {3.05/3.20,3.62/3.78,4.19/4.36,4.76/4.94,5.33/5.52,5.90/6.10} {
    \draw[stable]
       (\u,6.83) .. controls ({\u+.08},5.95) and ({\v-.08},5.18) .. (\v,4.32);
}

\foreach \u/\v in {3.05/3.20,3.62/3.78,4.19/4.36,4.76/4.94,5.33/5.52,5.90/6.10} {
    \fill[branchblue] (\u,6.83) circle (2.35pt);
    \fill[branchorange] (\v,4.32) circle (2.35pt);
}
\node[branchblue,above=2pt] at (4.55,6.87)
    {$B'\subset J_1^+\cap\widetilde X$};
\node[branchorange,below=2pt] at (4.83,4.28)
    {$h^-(B')\subset J_2^+$};

\draw[branchorange,->,line width=1.25pt]
    (4.19,6.78) .. controls (4.28,5.91) and (4.29,5.18) .. (4.36,4.38);
\node[fill=white,text=branchorange,inner sep=1.5pt] at (4.28,5.58) {$h^-$};
\node[softgray,rotate=82] at (6.48,5.58)
    {$\widetilde W^-(w';\widetilde g)$};

\node[note,branchorange] at (7.10,3.45)
    {\\
     $h^-(B')\cap\widetilde X=\varnothing$};

\draw[panel] (10.20,2.80) rectangle (16.35,8.15);
\node[anchor=west,font=\bfseries] at (10.45,7.85)
    {(b) Image under $\widetilde\pi$};

\draw[maparrow] (9.38,5.70) -- node[above] {$\widetilde\pi$} (10.10,5.70);

\draw[black!55,line width=.9pt] (10.72,5.46) -- (15.86,5.46);
\draw[targetgreen,line width=3.2pt,line cap=round] (11.38,5.46) -- (15.18,5.46);
\draw[branchorange,line width=1.15pt,dashed] (11.38,5.46) -- (15.18,5.46);
\fill[targetgreen] (13.28,5.46) circle (2.8pt);
\node[above=5pt] at (13.28,5.46)
    {$I^u\subset\widetilde W^u(\widetilde\pi(\widetilde x);\widetilde f)$};
\node[note] at (13.28,4.55)
    {$\widetilde\pi(J_1^+)=\widetilde\pi(J_2^+)\subset I^u$};
\node[note] at (13.28,3.65)
    {$\widetilde\pi(w')=\widetilde\pi(h^-(w'))$};

\draw[panel] (1.50,-1.10) rectangle (15.00,2.15);
\node[anchor=west,font=\bfseries] at (1.75,1.83)
    {(c) After a backward iterate};

\draw[maparrow] (6.20,2.70) -- node[right] {$\widetilde g^{-N}$} (6.20,2.22);

\draw[black!75,line width=1.15pt] (2.40,0.62) -- (14.10,0.62);
\draw[black!75,line width=.8pt] (2.40,0.47) -- (2.40,0.77);
\draw[black!75,line width=.8pt] (14.10,0.47) -- (14.10,0.77);
\node[anchor=west,above=6pt] at (2.40,0.62) {$\xi^+(x_N)$};

\foreach \x in {4.45,4.82,5.19,5.56,5.93,6.30} {
    \fill[branchblue] (\x,0.62) circle (2.35pt);
}
\foreach \x in {9.00,9.37,9.74,10.11,10.48,10.85} {
    \fill[branchorange] (\x,0.62) circle (2.35pt);
}
\node[branchblue,align=center,below=7pt] at (5.38,0.62)
    {$\widetilde g^{-N}(B')\subset\widetilde X$};
\node[branchorange,align=center,below=7pt] at (9.92,0.62)
    {$\widetilde g^{-N}(h^-(B'))\subset\widetilde M\setminus\widetilde X$};

\node[note,anchor=east] at (14.55,1.50)
    {$\displaystyle
      \mu_{x_N}^{\xi^+}\!\left(\xi^+(x_N)\setminus\widetilde X\right)>0$\\[-1pt]
     \textit{contradiction}};

\end{tikzpicture}

\caption{ The fold-holonomy mechanism in \Cref{prop:injectivity.unstable}. Stable holonomy sends the positive conditional-measure set \(B'\subset J_1^+\cap\widetilde X\) to a positive conditional-measure subset of \(J_2^+\setminus\widetilde X\). After a backward iterate, both sets lie in the same atom of \(\xi^+\), giving a contradiction.} \label{fig:fold.holonomy}
\end{figure}

Iterating backwards if necessary, there exists $N>0$ and $x_N=\widetilde g^{-N}(\widetilde x)$, such that $g^{-N}(B')\cup g^{-N}(h^-(B'))\subset \xi^+(x_N)$. 
Therefore, for every $\widetilde x\in \widetilde X$ such that there is $\widetilde y\in \widetilde W^+(\widetilde x;\widetilde g)$ with $\widetilde\pi(\widetilde x)=\widetilde\pi(\widetilde y)$, there exists $N>0$ such that
$$\mu^{\xi^+}_{x_N}(\xi^+(x_N)\setminus \widetilde X)>0.$$
Let
\[
E^+=
\left\{
z\in \widetilde M:
\mu_z^{\xi^+}\bigl(\xi^+(z)\setminus \widetilde X\bigr)>0
\right\}.
\]
Since $\widetilde X$ is a conull set, disintegration along $\xi^+$ gives $\mu(E^+)=0$.

The preceding argument shows that if $\widetilde x\in \widetilde X$ and there exists
$\widetilde y\in \widetilde W^+(\widetilde x;\widetilde g)\setminus\{\widetilde x\}$ such that $\widetilde \pi(\widetilde x)=\widetilde\pi(\widetilde y)$, then
$\widetilde g^{-N}(\widetilde x)\in E^+$ for some $N\geq0$. Consequently, the set of such
points $\widetilde x$ is contained in
\[
\bigcup_{N\geq0}\widetilde g^N(E^+),
\]
which has zero measure. Removing this set and all its iterates
from $\widetilde X$, we obtain an invariant full-measure set for which
\[
\widetilde \pi^{-1}(\widetilde\pi(x))\cap\widetilde W^+(\widetilde x;\widetilde g)=\{\widetilde x\}.
\]
Applying the same argument to $\widetilde g^{-1}$ proves the stable assertion.
\end{proof}

\begin{corollary}\label{cor:isomorphism}
The map $\pi: (g,M,\mu)\to (f,M,\pi_*\mu)$ is a metric isomorphism.
\end{corollary}

\begin{proof}
Let $X\subset M$ be the full measure set obtained in  Propositions \ref{prop:injectivity.unstable} and \ref{prop:leafwise.surjectivity}.\par

We claim that $\pi$ restricted to $X$ is injective. Assume $\pi(x)=\pi(y)$, with $x,y\in X$ and consider lifts 
$\widetilde x, \widetilde y$ of $x,y$ such that $\widetilde\pi(\widetilde x)=\widetilde \pi(\widetilde y)=\widetilde q$. Then , $\widetilde q\in \widetilde W^u(\widetilde q;\widetilde f)\cap \widetilde W^s(\widetilde q;\widetilde f)$. \par
\Cref{rem:finite.intersection} implies that there is a point
$$\widetilde z\in\widetilde W^+(\widetilde x;\widetilde g)\transv_{top} \widetilde W^-(\widetilde y;\widetilde g).$$
\begin{align*}
&\widetilde z\in \widetilde W^+(\widetilde x;\widetilde g)\Rightarrow \sup_{n\leq 0}d(\widetilde g^n(\widetilde x),\widetilde g^n(\widetilde z))<\infty\\
&\widetilde z\in \widetilde W^-(\widetilde y;\widetilde g)\Rightarrow \sup_{n\geq 0}d(\widetilde g^n(\widetilde y),\widetilde g^n(\widetilde z))<\infty,
\end{align*}
Since for every $n$, $\widetilde\pi(\widetilde g^n(\widetilde x))=\widetilde\pi(\widetilde g^n(\widetilde y))$, using that $\widetilde \pi$ is at bounded distance from the identity, we deduce
\[
    \sup_{n\in\mathbb Z}d(\widetilde g^n(\widetilde x),\widetilde g^n(\widetilde z))<\oo
\]
By Proposition \ref{prop.semiconj}, $\widetilde \pi(\widetilde z)=\widetilde\pi(\widetilde x)=\widetilde\pi(\widetilde y)$. Hence $\pi(x)=\pi(y)=\pi(z)$. By \Cref{prop:injectivity.unstable}, $x=z$, and  $y=z$. Then $x=y$.
\end{proof}


Putting everything together, we have concluded the proof of \Cref{mainthm:isomorphism.mu}.\par\smallskip
The proof of \Cref{rem:conservative} for $C^{1+\alpha}$ conservative $g$ follows in a very similar way, using absolute continuity of the conditional measures instead of Ben Ovadia's result.

\begin{remark}
If $g\in C^\infty$ and $\mu=m$, then $\pi$ is differentiable in the Whitney sense \(m\)-almost everywhere.
\end{remark}
The proof of this fact follows exactly as in the proof of Theorem $1.3 (ii)$, Section $5$ of \cite{de_la_Llave_1992}. In that part of the argument, the homeomorphism property of the semiconjugacy $\pi$ is not used; it only uses the identity $\pi_*m=\nu$  with $\nu$ absolutely continuous, and the fact that the  Lyapunov exponents are the same for both systems, which follows from the fact that $h_m(g)=h_\nu(f)$, and Pesin's entropy formula \cite{pesin77} applied to $g,g^{-1}, f$ and $f^{-1}$. That these are the only required conditions is stated explicitly in \cite[Page 313]{de_la_Llave_1992}.

\subsection{Monotonicity of \texorpdfstring{$\pi$}{pi} on invariant manifolds}\label{sec:monotonicity}
Let $g$ be in the hypotheses of \Cref{mainthm:fiber} and $\mu$ be a $g$-invariant entropy-large measure. 
\begin{claim}\label{claim:correspondence.manifolds} Let $X$ be the full-measure set obtained in \Cref{prop:injectivity.unstable}. Then, for every $\widetilde x,\widetilde y\in \widetilde X$, if $\widetilde\pi(\widetilde y)\in\widetilde W^{u/s}(\widetilde \pi(\widetilde x);\widetilde f)$, then $\widetilde y\in \widetilde W^\pm(\widetilde x;\widetilde g)$.
\end{claim}
\begin{proof}
Let $\widetilde x,\widetilde y\in \widetilde X$ satisfy 
$$\widetilde\pi(\widetilde y)\in\widetilde W^u(\widetilde\pi(\widetilde x);\widetilde f)$$
By \Cref{rem:finite.intersection}, there exists $$\widetilde z\in \widetilde W^+(\widetilde x;\widetilde g)\cap \widetilde W^-(\widetilde y;\widetilde g).$$
By \Cref{prop.semiconj}, 
$$\widetilde\pi(\widetilde z)=\widetilde\pi(\widetilde y).$$
Then \Cref{prop:injectivity.unstable} implies $\widetilde z=\widetilde y$.
\end{proof}
\begin{proposition}\label{prop:monotonicity}
    There exists a full-measure set $X$ such that for every lift $\widetilde x\in \widetilde X$
$\widetilde\pi$ is monotonous when restricted to $\widetilde W^+(\widetilde x;\widetilde g)$ and $\widetilde W^-(\widetilde x;\widetilde g)$.
\end{proposition}
\begin{proof}
    Let $\nu=\MME=\pi_*\mu$. We choose the full-measure set $X$ satisfying \Cref{prop:injectivity.unstable}

and so that

$$
\nu_y^u(\pi(X))=1
$$

for every $y\in\pi(X)$, where $\nu_y^u$ denotes the unstable conditional measure of $\nu$ at $y$. Since these conditional measures have full support, $\widetilde\pi(\widetilde X\cap \widetilde W^+(\widetilde x;\widetilde g))$ is dense in $\widetilde W^u(\widetilde \pi(x);\widetilde f)$ for every $\widetilde x\in \widetilde X$ that is a lift of $x\in X$. This follows from \Cref{claim:correspondence.manifolds} above. 

We claim that the restriction of $\widetilde\pi$ to $\widetilde W^+(\widetilde x;\widetilde g)$ is monotone for every $\widetilde x\in \widetilde X$. Otherwise, there would exist points

$$
a<c<b
$$

in $\widetilde W^+(\widetilde x;\widetilde g)$ such that

$$
\widetilde \pi(a)=\widetilde\pi(b)\neq\widetilde\pi(c).
$$

Let $I$ be the open interval in $\widetilde W^u(\widetilde\pi(\widetilde x);\widetilde f)$ with endpoints $\widetilde\pi(a)$ and $\widetilde\pi(c)$. For every $y\in I$, the intermediate value theorem, applied first to $[a,c]$ and then to $[c,b]$, gives two distinct points $z_1\in(a,c)$ and $z_2\in(c,b)$ such that

$$
\widetilde\pi(z_1)=\widetilde\pi(z_2)=y.
$$

Consequently,

$$
I\cap\widetilde\pi(\widetilde X\cap\widetilde W^+(\widetilde x;\widetilde g))=\varnothing,
$$

because every point of $\widetilde\pi(\widetilde X\cap \widetilde W^+(\widetilde x;\widetilde g))$ has a singleton fiber. This contradicts the density of $\widetilde \pi(\widetilde X\cap \widetilde W^+(\widetilde x;\widetilde g))$ in $\widetilde W^u(\widetilde\pi(\widetilde x);\widetilde f)$. Therefore, $\widetilde\pi$ is monotone on $\widetilde W^+(\widetilde x;\widetilde g)$.
\end{proof}
The monotonicity for $\pi$ on invariant manifolds in the conservative case of \Cref{rem:conservative} follows in the same way. 
\section{A criterion for cellularity}\label{sec:criterion}

In this part we establish a general criterion to prove cellularity of the fibers of $\pi$. We will use the monotonicity of $\widetilde \pi$ on lifted invariant manifolds to construct stable/unstable polygons in $\widetilde M$,  and trap each fiber inside a decreasing sequence of such polygons. \par

This dynamical machinery will be used to establish the existence of this sequence of polygons, assuming the hypotheses of 
\Cref{mainthm:fiber}.


\subsection{Adapted \texorpdfstring{$su$}{su}-polygons}
\label{ssub:adaptedpolygons}
Let $X$ be an invariant full-measure set on which the conclusions of Propositions \ref{prop:leafwise.surjectivity}, \ref{prop:injectivity.unstable}, and \ref{prop:monotonicity} hold simultaneously for stable and unstable manifolds. Let $\widetilde X$ be its lift to $\widetilde M$. Thus, for every $\widetilde x\in \widetilde X$, 
$$\widetilde\pi(\widetilde W^\pm(\widetilde x;\widetilde g))=\widetilde W^{u/s}(\widetilde\pi(\widetilde x);\widetilde f);$$
the restriction of $\widetilde\pi$ to each of these invariant manifolds is monotone, and 
$$\widetilde\pi^{-1}(\widetilde\pi(\widetilde x))\cap\widetilde W^\pm(\widetilde x;\widetilde g)=\{\widetilde x\}.$$
\par

We use the set $\widetilde X$ to construct polygons around each point of $\widetilde M$. An {\em $su$-polygon} is a closed disk whose boundary is formed by a finite number of stable and unstable segments. Each intersection between a stable and an unstable segment is called a {\em vertex}.

\begin{lemma}[$su$-polygons for $f$]\label{lem:su.polygons.pi.X}
For every \(\widetilde p\in\widetilde M\), there exists 
a neighborhood basis of leafwise convex $su$-polygons such that its vertices belong to $\widetilde\pi(\widetilde X)$.
\end{lemma}

\begin{proof}
Let \(l\) be the number of stable, equivalently unstable, prongs at
\(\widetilde p\). Using a sufficiently small prong chart centered at
\(\widetilde p\), choose \(2\ell\) segments
$$S^*_1,U^*_1,\dots,S^*_\ell,U^*_\ell$$
alternating between stable and unstable leaves, passing through points of
\(\widetilde\pi(\widetilde X)\), and bounding a cell \(P\) containing
\(\widetilde p\). This is the minimal number of segments that an $su$-polygon containing $\widetilde p$ can have.\par
For each $i$, choose small intervals
$I_i\subset U^*_i$ and $J_i\subset U^*_{i+1}$, close to the vertices joined
by the stable side between $U^*_i$ and $U^*_{i+1}$, such that the stable
holonomy
\[
h_i\colon I_i\longrightarrow J_i
\]
is defined and all its holonomy arcs are contained in $P$.

Set
\[
A_i=I_i\cap\pi(X).
\]
The set $A_i$ has full unstable conditional measure in $I_i$. Since the 
stable holonomy preserves the unstable conditional measure class,
\[
h_i(A_i)\cap J_i\cap\pi(X)
\]
has full unstable conditional measure in $J_i$. We may therefore choose
\[
y_i\in h_i(A_i)\cap J_i\cap\pi(X)
\]
and set
\[
x_i=h_i^{-1}(y_i).
\]
Then
\[
x_i,y_i\in\pi(X),
\]
and $x_i$ and $y_i$ are joined by a stable arc contained in $P$.

For each $i$, join $y_{i-1}$ to $x_i$ by the unstable arc contained in
$U^*_i$. If the intervals $I_i$ and $J_i$ are chosen sufficiently close
to the corresponding vertices of $P$, the curve
\[
\Gamma=
\bigcup_{i=1}^{l}
\left(
[y_{i-1},x_i]_{U^*_i}
\cup
[x_i,y_i]_{S^*_i}
\right)
\]
is an $su$-polygon contained in $P$ and containing $p$.
Moreover, all its vertices
\[
x_1,y_1,\ldots,x_\ell,y_\ell
\]
belong to $\pi(X)$. For future use, we rename $S^*_i$, $U^*_i$, so that 
$$\partial P=\bigcup_{i=1}^\ell S^*_i\cup U^*_i.$$
\end{proof}

\begin{lemma}[$su$-polygons for the fiber]\label{lem:su.polygons.X} For each $\widetilde p\in\widetilde M$ and for each $su$-polygon $P$ found in \Cref{lem:su.polygons.pi.X} there exists an $su$-polygon $Q$ such that
$$\widetilde\pi(\partial Q)=\partial P\quad\text{and}\quad \widetilde\pi^{-1}(\widetilde p)\cap Q\ne\emptyset$$
\end{lemma}
\begin{proof}

    Since the vertices of $P$ belong to $ Y=\widetilde\pi(\widetilde X)$, there are unique segments 
    $$[v_1,w_1]_-,[w_1,v_2]_+,\dots,[w_\ell,v_1]_+,$$
    such that $\widetilde\pi(v_i)$, $\widetilde\pi(w_i)$ are the vertices of $P$ and $\widetilde \pi[v_i,w_i]_-$, $\widetilde \pi[w_i,v_i]_+$ are the sides of $\partial P$. 
The uniqueness of the segments follows from the choice of the vertices and \Cref{claim:correspondence.manifolds}.    
    Call 
$$S_i=[v_i,w_i]_-\quad\text{and}\quad U_i=[w_i,v_{i+1}]_+,$$
\par
Since $$\widetilde\pi|_{S_i}\quad\text{and}\quad \widetilde\pi|_{U_i}$$
are not necessarily injective, it could happen that
$$\#S_i\cap U_j>1\quad\text{with }i\ne j. $$
 We will see that 
\begin{equation}\label{eq:partial.R}
\partial R=\bigcup_{i=1}^\ell S_i\cup U_i,    
\end{equation}
is a Jordan curve satisfying

\begin{equation}\label{eq:su.polygon.R}
\widetilde\pi(\partial R)=\partial P\quad\text{and}\quad \widetilde p\in \widetilde\pi(R_0), 
\end{equation}
where $R_0$ is the connected component bounded by $\partial R$.
\smallskip
For every $i=1,\dots,\ell$,
$$S_i\cap U_j$$
consists of a unique point if $j=i$ or $i-1$, adopting the notation $U_0=U_\ell$. This is a consequence of the fact that the vertices $v_i,w_i\in X$ and \Cref{prop:injectivity.unstable}. On the other hand, if $j\notin\{i,i-1\}$, then $\widetilde\pi(S_i)\cap\widetilde\pi(U_j)=S^*_i\cap U^*_j=\varnothing$. This implies that $\partial R$ is a Jordan curve satisfying $\deg(\pi|\partial R)=\pm 1$, and then $\widetilde\pi(\partial R)=\partial P$. 
Let $Q=\overline{R_0}$ the closure of the connected component of $\widetilde M\setminus\partial R$. Since $\deg(\pi|\partial Q)=\pm 1$, and  $\widetilde\pi(\partial Q)=\partial P$, we have $$P\subset \widetilde \pi(Q).$$
\end{proof}

\begin{lemma}[Fiber trapping]\label{lem:fiber.trapping}
For the $su$-polygons $P$ and $Q$ constructed in Lemmas \ref{lem:su.polygons.pi.X} and \ref{lem:su.polygons.X}, one has
\[
\widetilde\pi^{-1}(\widetilde p)\subset Q.
\]
\end{lemma}

\begin{proof}
Set
\[
F=\widetilde{\pi}^{-1}(\widetilde{p}).
\]
By \Cref{lem:su.polygons.X}, there exists
\[
r\in F\cap\operatorname{int}(Q),
\]
since $\widetilde{\pi}(\partial Q)=\partial P$ and
$\widetilde{p}\in\operatorname{int}(P)$.

Suppose, by contradiction, that there exists
\[
q\in F\setminus Q.
\]
Extend the stable and unstable sides of $\partial Q$ to complete invariant
manifolds $L_i$, $i=1,\dots,2\ell$. 
If $r$ and $q$ are on the same connected component of $\widetilde M\setminus L_i$ for every $i$, then $r$ and $q$ belong to $Q$. Then there is a line, say $L$, that separates $r$ from $q$.

Let
\[
K=\sup_{z\in\widetilde M}
d\bigl(z,\widetilde{\pi}(z)\bigr)<\infty.
\]
Since $r,q\in F$, for every $n\in\mathbb Z$,
\[
d\bigl(\widetilde g^{\,n}(r),\widetilde f^{\,n}(\widetilde p)\bigr)
\leq K,
\qquad
d\bigl(\widetilde g^{\,n}(q),\widetilde f^{\,n}(\widetilde p)\bigr)
\leq K.
\]
Moreover, $\widetilde g^{\,n}(L)$ separates
$\widetilde g^{\,n}(r)$ from $\widetilde g^{\,n}(q)$. Consequently,
\[
d\bigl(\widetilde f^{\,n}(\widetilde p),
       \widetilde g^{\,n}(L)\bigr)
\]
is uniformly bounded in $n$. Since $\widetilde{\pi}$ is at bounded distance
from the identity and
\[
\widetilde{\pi}\bigl(\widetilde g^{\,n}(L)\bigr)
=
\widetilde f^{\,n}\bigl(\widetilde{\pi}(L)\bigr),
\]
it follows that
\[
d\left(
  \widetilde f^{\,n}(\widetilde p),
  \widetilde f^{\,n}\bigl(\widetilde{\pi}(L)\bigr)
\right)
\]
is also uniformly bounded.

Since $\widetilde f$ is infinitely expansive on $\widetilde M$, this would imply that $\widetilde p\in\widetilde\pi(L)$, but $\widetilde p$ does not belong to the supporting leaf
$\widetilde{\pi}(L)$. Therefore such a point $q$ cannot exist, and hence
\[
\widetilde{\pi}^{-1}(\widetilde p)\subset Q.
\]

\end{proof}

\begin{proposition}\label{prop:adapted.neighborhood.basis}
Let
\[
F=\widetilde\pi^{-1}(\widetilde p).
\]
Let $P_n$ be a decreasing sequence of $su$-polygons as obtained in \Cref{lem:su.polygons.pi.X} such that
$$\{\widetilde p\}=\bigcap_n P_n.$$
Let $Q_n$ be the $su$-polygon obtained from $P_n$ as in \Cref{lem:su.polygons.X}. Then,
$$F=\bigcap_nQ_n.$$
\end{proposition}
\begin{proof}
From \Cref{lem:fiber.trapping} we know that
$$F\subset \bigcap_n Q_n.$$
Let $$y\in\bigcap_n Q_n.$$ Suppose $\widetilde\pi(y)\ne\widetilde p$. Then, there are two disjoint $su$-polygons $P$ and $P'$, as found in \Cref{lem:su.polygons.pi.X}, such that $\widetilde p\in P$ and $\widetilde\pi(y)\in P'$. Call $Q$ and $Q'$ the corresponding $su$-polygons found in \Cref{lem:su.polygons.X}. We may assume $y\in \inte Q$ and $q=\pi(y)\in\inte P'$. Then, \Cref{lem:fiber.trapping} implies that $y\in\inte Q'$.
Now, 
$$\partial Q\cap \partial Q'\subset \widetilde\pi^{-1}(\widetilde\pi(\partial Q)\cap\widetilde\pi(\partial Q'))\subset \widetilde\pi^{-1}(\partial P\cap\partial P')=\varnothing.$$
Two Jordan disks with disjoint boundaries and intersecting interiors must be nested. \par
Assume $$Q'\subset\inte Q.$$
Take a vertex $v'$ of $Q'$ and set $L'=\widetilde W^+(v';\widetilde g)$. Let $[a,b]_{L'}$ be the component of $L'\cap Q$ containing $v'$. Since $L'$ is proper, $a,b\in\partial Q$. Hence, 
$$\widetilde\pi(a),\widetilde\pi(b)\in\partial P\cap\widetilde\pi(L').$$
By leafwise convexity of $P$, the interval in $\widetilde\pi(L')$ between these two points is contained in $P$. Since $\widetilde\pi$ is monotone on $L'$ and $a<v'<b$, $$\widetilde\pi(v')\in P.$$
But $\widetilde\pi(v')$ is a vertex of $P'$, contradicting $P\cap P'=\varnothing$. The opposite nesting is treated symmetrically. Thus, $Q\cap Q'=\varnothing$, contradicting that $y\in Q\cap Q'$. 
\end{proof}
\begin{lemma}\label{lem:nested} The sequence of $su$-polygons obtained in \Cref{lem:su.polygons.pi.X} can be chosen so that the corresponding sequence of $su$-polygons $Q_n$ obtained from
$P_n$ in \Cref{lem:su.polygons.X} is a nested sequence.

\end{lemma}

\begin{proof}
From \Cref{lem:fiber.trapping} it follows that
\[
\widetilde{\pi}^{-1}(\operatorname{int}(P_n))
\subset \operatorname{int} Q_n.
\]
Choose
\[
P_{n+1}\subset \operatorname{int}(P_n).
\]
Then
\[
\partial Q_{n+1}
\subset \widetilde{\pi}^{-1}(\partial P_{n+1})
\subset \operatorname{int} Q_n.
\]
\end{proof}
\smallskip 
 
Therefore $F=\widetilde\pi^{-1}(\widetilde p)$ is cellular for every $\widetilde p\in\widetilde M$. The same thus holds in $M$:
$\pi^{-1}(p)$ is cellular for every $p\in M$.
This finishes the proof of \Cref{mainthm:fiber}.

\subsection*{Acknowledgments} P.D.C., J.R.-H., and R.U. thank Penn State University for its hospitality, where this project was initiated. P.D.C. also thanks SUSTech for its hospitality, where part of this work was carried out.

\section*{Declaration of generative AI and AI-assisted technologies in the manuscript preparation process}

During the preparation of this manuscript, the authors used ChatGPT (OpenAI, GPT-5.6 Sol) for assistance with language editing, LaTeX formatting, and identifying passages or arguments requiring clarification or further verification. The authors reviewed and revised all output, independently verified the mathematical arguments and references, and take full responsibility for the content of the manuscript.

 \bibliographystyle{alpha}
 \bibliography{references}

 \end{document}